\documentclass[a4paper,reqno, 11pt]{amsart}
\usepackage[utf8]{inputenc}
\usepackage{tikz}
\usepackage{tikz-cd}
\usepackage[top=2.5cm,bottom=2.5cm,left=2.2cm,right=2.2cm]{geometry}
\usepackage{graphicx}
\usepackage{mathtools}

\usepackage{hyperref}

\usepackage[dvips]{epsfig}
\usepackage{color}
\usepackage{verbatim}
\usepackage{amsgen, amstext,amsbsy,amsopn,amssymb, amsthm, mathptmx, amsfonts,amssymb,amscd,amsmath,nicefrac,euscript,enumerate,url,verbatim,calc,framed}
\usepackage[all]{xy}
\def\red#1{{\color{red}#1}}

\theoremstyle{plain}
\newtheorem{thm}{\bf Theorem}[section]

\newtheorem{prop}[thm]{\bf Proposition}
\newtheorem{lem}[thm]{\bf Lemma}

\newtheorem{setup}{\bf Setup}

\theoremstyle{definition}

\newtheorem{notation}[thm]{\bf Notation}

\theoremstyle{remark}
\newtheorem{rem}[thm]{\bf Remark}
\newtheorem{quest}[thm]{\bf Question}

\newtheorem{note}[thm]{\bf Note}

\DeclareMathOperator{\ass}{Ass}

\DeclareMathOperator{\grade}{grade}

\DeclareMathOperator{\Min}{Min}

\DeclareMathOperator{\pf}{Pf}
\DeclareMathOperator{\reg}{reg}

\DeclareMathOperator{\Sym}{Sym}

\DeclareMathOperator{\dev}{dev}

\def\height{\operatorname{ht}}

\def\NN{\mathbb{N}}

\newsavebox\foobox

\begin{document}

\title{\textbf{Free Resolutions of Symmetric 
Algebras of Ideals with Deviation Two}}

\author[Neeraj Kumar]{Neeraj Kumar}
\address{Department of Mathematics, Indian Institute of Technology Hyderabad, Kandi, Sangareddy - 502285, Telangana, INDIA}
\email{neeraj@math.iith.ac.in}

\author[Aniruddha Saha]{Aniruddha Saha}
\address{Department of Mathematics, Indian Institute of Technology Hyderabad, Kandi, Sangareddy - 502285, Telangana, INDIA}
\email{ma20resch11001@iith.ac.in}

\author[Chitra Venugopal]{Chitra Venugopal}
\address{Chennai Mathematical Institute, Siruseri, Tamil Nadu 603103. India}
\email{chitrav@cmi.ac.in}

\date{\today}

\begin{abstract}
For a graded ideal $I$ in a graded ring, the {deviation} of $I$ is defined as the difference between the minimal number of generators of $I$ and its grade. In this article, we provide bigraded free resolutions of the symmetric algebras for specific classes of ideals of deviation two. Additionally, we study the regularity of powers of deviation two ideals generated by $d$-sequences. In particular, we examine the powers of Huneke–Ulrich ideals and find bounds on their regularity.  
\end{abstract}

\maketitle

\vspace{-1cm}

\section*{Introduction}
Symmetric and Rees algebras of graded ideals are among the most studied bigraded algebras as they play an important role in commutative algebra and algebraic geometry. The Rees algebra, in particular, provides an algebraic realization of the concept of blowing up a variety along a subvariety. An important class of ideals in this context are the ideals of linear type, since for these ideals the symmetric algebra and the Rees algebra are isomorphic. Minimal free resolutions are one of the main tools which helps in studying their structures. However, determining these resolutions for arbitrary ideals is generally a difficult problem. Imposing additional conditions on the generators of the ideal often leads to a better understanding of the associated homological invariants.

\vspace{2mm}

Let $A$ be a graded ring, and let $I \subseteq A$ be a graded ideal. The \textit{deviation} of $I$, denoted by $\dev(I)$, is the difference between the minimal number of generators of $I$ and its grade. For ideals of deviation zero, that is, complete intersections, Micali in \cite{M1964}, showed that the corresponding Rees algebra is isomorphic to the symmetric algebra. Moreover, Micali and Roby in \cite{MR1971} provide an explicit description of the defining equations of the Rees algebra as the $2 \times 2$ minors of a $2 \times n$ matrix, where $n$ is the minimal number of generators of the ideal. Thus, the Eagon–Northcott complex \cite{EN, Eisenbud_text} provides a free resolution of the Rees algebra of ideals of deviation zero.

\vspace{2mm}

The approximation complex $\mathcal{Z}_\bullet $ (See Complex (\ref{eq:ApproxComplex})) is a foundational framework for the study of blowup algebras; under suitable hypotheses, it provides free resolutions of symmetric algebras $\Sym_A(I)$ (cf.\cite{HSV1, HSV2}). Its structural drawback is equally well known; $\mathcal{Z}_\bullet $ is built from Koszul cycles $Z_i $, which are typically \emph{not} free, and this nonfreeness propagates into the higher modules of the complex, making explicit resolutions difficult to read off. Motivated by \cite[Theorem~5.8]{HSV3}, Cid-Ruiz, Polini and Ulrich in \cite{Ulrich&Polini} refine $ \mathcal{Z}_\bullet $ to a \emph{halfway} resolution.
 If $g=\dev(I)>0 $, the new complex agrees with $\mathcal{Z}_\bullet $ in its initial portion and replaces the last $g-1 $ positions by free modules. Moreover, the acyclicity of $\mathcal{Z}_\bullet$ guarantees the acyclicity of the new complex
$\text{\cite[Proposition~4.5]{Ulrich&Polini}}$. This refinement is significant because computing free resolutions of symmetric algebras is, in general, a steep challenge. In this article we focus on the case $g=2$; in this setting $\mathcal{C}_\bullet $ helps in identifying the tail end of Complex (\ref{eq:ApproxComplex}) by free modules, see Complex (\ref{NewComplex}).

\vspace{2mm}

For positive deviation $g=\dev(I)\ge 1$, a free resolution of $\Sym_A(I)$ ultimately depends on controlling the Koszul homology $ H_i\big(K_\bullet(\mathbf f;A)\big), \; 1\le i\le g-1,$ for a minimal generating sequence $\mathbf f $ of $I$. Observe that non-vanishing Koszul homology modules need not be free or even Cohen–Macaulay, and determining their minimal resolutions is a subtle and intricate task. As the deviation $g$ increases, the construction of a free resolution of $\Sym_A(I)$ becomes increasingly challenging. The initial case $g=1$ (almost complete intersections) is relatively moderate. In that setting, one needs only the first Koszul cycle $Z_1(\mathbf f;A)$, which is resolved by a truncation of the minimal free resolution of $I$. Explicit free resolutions of $\Sym_A(I)$ are studied for both equigenerated and non-equigenerated ideals, see \cite{Cazalet, NKCV3, Ulrich&Polini}. 

\vspace{2mm}

The case $g=2$ (deviation two) is the first place where substantially new behavior occurs. Under a strong structural hypothesis, namely, $I$ is a perfect ideal of deviation two and the approximation complex $\mathcal{Z}_\bullet$ is acyclic: 
Cid-Ruiz, Polini and Ulrich in \cite[Theorem~4.11(ii)]{Ulrich&Polini} construct an explicit free resolution of the symmetric algebra $\operatorname{Sym}_A(I)$. In this setting, there is a canonical identification $
H_2\!\big(K_\bullet(\mathbf f;A)\big)\ \cong\ \omega_{A/I}\;\text{(up to a shift)}$, so the second Koszul homology is resolved by dualizing a minimal free resolution of $A/I$. When $I$ is not perfect, constructing a free resolution of $\Sym_A(I)$ in the deviation two case becomes challenging. One must first develop homological methods to resolve the second Koszul homology module.

\vspace{2mm}

A primary objective of this paper is to develop homological constructions to resolve the second Koszul homology module and apply them to obtain a free resolution of $\Sym_A(I)$. In this article, we study the symmetric algebra of ideals of deviation two that are not necessarily perfect. In Section \ref{Section3}, we construct two families of deviation two ideals starting from a given ideal and develop the homological tools that will be used to resolve the corresponding symmetric algebras in the subsequent section. This involves studying the second Koszul homology modules of these ideals and providing an explicit description of their free resolutions in terms of the free resolution of the second Koszul homology module of the base ideal (see Lemma \ref{Res_of_H2} and Proposition \ref{H1-regular}). In Section \ref{Section4},  Theorem \ref{RES}, we provide explicit bigraded free resolution of the symmetric algebra for these classes of deviation two ideals. These ideals are generated iteratively, and their construction has combinatorial motivation.

\vspace{2mm}

Another important theme of this article is the study of the asymptotic properties of powers of ideals. Among various homological invariants, the Castelnuovo–Mumford regularity of an ideal is particularly significant, as it provides bounds on the degrees of its syzygies. The regularity of powers of complete intersections, which are ideals of deviation zero, is well understood. The authors in \cite{DBDE75} show that the powers of a complete intersection ideal are determinantal ideals, and hence the minimal free resolution of these powers can be obtained using the Eagon-Northcott complex. Moreover, in the case of a graded complete intersection in a polynomial ring, the authors in \cite{GV} provide the explicit graded Betti numbers of its powers. Similarly, the regularity of powers of almost complete intersections (ideals of deviation one) has been studied in \cite{Shen&Zhu, NKCV3}. 

\vspace{2mm}

In this article, we study the regularity of powers to ideals of deviation two. In Theorem~\ref{powers_aci}, we provide expressions for the regularity of powers of deviation two ideals that are generated by $d$-sequences and satisfy certain structural conditions. Another important class of deviation two ideals considered in Section \ref{Section2} is the Huneke-Ulrich ideals (see Subsection~\ref{HU Ideals} for the definition). These are non-trivial Gorenstein ideals of deviation two and grade at least four \cite{Huneke&Ulrich}, and they exhibit several notable properties, including being generated by $d$-sequences, as discussed in Proposition \ref{HU_prop}. In Theorem\ref{Reg_powers_HU}, we establish bounds on the regularity of powers of this class of ideals.

\vspace{2mm}

Throughout the article, $K$ denotes a field of characteristic zero and $B$ denotes a polynomial ring over the field $K$.

\section{Preliminaries} \label{prelim}

In this section, we recall some definitions and results that will be used throughout the article. We begin by reviewing the concept of $d$-sequences, which play a central role in the study of ideals of linear type.

\vspace{2mm}

Let $\{ \bf{a} \}$ denote a sequence $\{ a_1, \ldots, a_n \}$ in a graded ring $A$ and $N$ be a finitely generated graded $A$-module. Let $I=\langle a_1, \ldots, a_n\rangle$, $I_i=\langle a_1, \ldots, a_i \rangle$ and set $I_0$ to be the zero ideal. Then the sequence $\{\bf{a}\}$ is a \textit{$d$-sequence} for $N$ if for each integer $i=1, \ldots, n$ and each integer $k=i, \ldots, n$, the following conditions hold (cf. \cite{CH1})
\begin{enumerate}[(a)]
    \item $a_i \notin \langle a_1, \ldots, \hat{a_i}, \ldots, a_n \rangle$,
    \item $(I_{i-1}N:_Na_ia_k) = (I_{i-1}N:_Na_k)$.
\end{enumerate}

This notion generalizes regular sequences and has significant implications for the structure of Rees and symmetric algebras. Numerous examples of $d$-sequences can be found in \cite{CH1}. An important result, due to Huneke \cite{CH} and independently proved by Valla \cite{Valla}, states that the ideals generated by $d$-sequences are of linear type, that is, their Rees algebra coincides with their symmetric algebra.

\begin{note} \label{properties}
The following are some properties of $d$-sequences and regular sequences that will be used in this article.
\begin{enumerate} [(i)]
\item \label{p11} Let $\{f_1, \ldots, f_n\}$ be a $d$-sequence in a ring $A$. Then, $(\langle f_1, \ldots, f_{n-1} \rangle:f_n^s)=(\langle f_1, \ldots, f_{n-1} \rangle:f_n)$ for all $s \geq 1$.
\item \label{p12}  Let $\{f_1, \ldots, f_n\}$ be a $d$-sequence in a ring $A$, and let $I= \langle f_1, \ldots, f_n \rangle$. Then, for $s \geq 1$, $${(\langle f_1, \ldots, f_{i-1} \rangle + I^s:f_i)}=(\langle f_1, \ldots, f_{i-1} \rangle:f_i) + I^{s-1}$$ for all $1 \leq i \leq n$ (\cite[Observation 2.4]{Shen&Zhu}).
\item \label{p2} If $\{f_1, \ldots, f_n\}$ forms a regular sequence, then $(\langle f_1, \ldots, f_{i-1} \rangle:f_i) = \langle f_1, \ldots, f_{i-1} \rangle$ for all $1 \leq i \leq n$.
\end{enumerate}
\end{note}

We now recall a key lemma concerning the Castelnuovo--Mumford regularity in short exact sequences. This result will be applied repeatedly in Section \ref{Section2} to obtain regularity bounds, and in some cases, to derive equalities.

\begin{lem} [Regularity Lemma] \label{reg_lemma} \cite[Lemma 3.1]{Hoa&Tam}
Let $0 \rightarrow M \rightarrow N \rightarrow P \rightarrow 0$ be a short exact sequence of finitely generated graded $A$-modules. Then, the following holds.
\begin{enumerate}[\rm (i)]
    \item \label{r1} If  $\reg(M) \neq \reg(P)+1,$ then $\reg(N)=\max  \{ \reg(M),\reg(P)\}$.
    \item \label{r2} If  $\reg(N) \neq \reg(P),$ then $\reg(M)=\max  \{ \reg(N),\reg(P)+1\}$.
    \item \label{r3} If  $\reg(M) \neq \reg(N),$ then  $\reg(P)=\max  \{ \reg(M)-1,\reg(N)\}$. 
\end{enumerate}
\end{lem}

\section{Regularity of powers of ideals of deviation two generated by $d$-sequences} \label{Section2}

In this section, we give expressions for the regularity of powers of deviation two ideals generated by $d$-sequences, similar to those obtained for deviation zero ideals in \cite{DBDE75,GV} and deviation one ideals in \cite{Shen&Zhu,NKCV3}.

\begin{setup} \label{s1}
Let $A$ be a standard graded $K$-algebra and $I$ be a homogeneous ideal of $A$ generated by a $d$-sequence $\{ f_1, \ldots,f_{n}\}$ in degrees $d_1, \ldots, d_n$ with $\{f_1, \ldots,f_{n-2}\}$ being a regular sequence. Assume that $d_n$ is the maximum of the degrees $d_i, i=1, \ldots, n$ and $J=\langle f_1, \ldots,f_{n-1}\rangle $.
\end{setup}

\begin{lem} \label{lem1}
Consider the notation and assumptions as in Setup \ref{s1}. Then, for $i=0, \ldots,n-1$,
\begin{enumerate}[\rm (i)]
\item If $\reg(A/I) < \reg(A/J)$, then,
$$ \reg \left( \dfrac{A}{\langle f_1, \ldots,f_i\rangle+I^2} \right)=\reg(A/J)+ d_n - 1.$$
\item If $\reg(A/I) > \reg(A/J)$, then,
$$ \reg \left( \dfrac{A}{\langle f_1, \ldots,f_i\rangle+I^2} \right)=\reg(A/I)+ d_n.$$
\end{enumerate}
\end{lem}

\begin{proof}
Consider the following short exact sequence,

$$0 \longrightarrow \dfrac{A}{J:f_{n}}(-d_n) \stackrel {.f_{n}}
    \longrightarrow \dfrac{A}{J} \longrightarrow \dfrac{A}{I} \longrightarrow 0$$    

\begin{enumerate}[\rm (i)]
\item Assume that $\reg(A/I) < \reg(A/J)$. Then, 
\begin{equation} \label{eq2}
\reg \left( \dfrac{A}{J:f_{n}} \right) = \reg \left( \dfrac{A}{J} \right) - d_n.
\end{equation}

The result is now proved by descending induction on $i$. For the base case $i=n-1$, consider the following short exact sequence

$$0 \longrightarrow \dfrac{A}{(J:f_{n}^2)}(-2d_n) \stackrel {.f_{n}^2}
    \longrightarrow \dfrac{A}{J} \longrightarrow \dfrac{A}{J+\langle f_{n}^2 \rangle} \longrightarrow 0.$$

As a consequence of Note \ref{properties}(\ref{p11}), we have $\reg \left( \dfrac{A}{J:\langle f_{n}^2 \rangle} \right)=\reg \left( \dfrac{A}{J:\langle f_{n} \rangle} \right)= \reg \left( \dfrac{A}{J} \right) - d_n$. This implies that $\reg \left( \dfrac{A}{J:\langle f_{n}^2 \rangle}(-2d_n) \right)=\reg \left( \dfrac{A}{J} \right) +d_n$. Therefore, from Lemma \ref{reg_lemma}(\ref{r3}), $\reg \left( \dfrac{A}{J+\langle f_{n}^2 \rangle} \right)=\reg \left( \dfrac{A}{J} \right) +d_n-1$. Now, assume that $i \leq n-2$ and the statement holds for $i$. Consider the following short exact sequence,
\begin{equation} \label{EQ1}
0 \longrightarrow \dfrac{A}{(\langle f_1, \ldots,f_{i-1}\rangle+I^2):f_i}(-d_i) \stackrel {.f_{i}}
    \longrightarrow \dfrac{A}{\langle f_1, \ldots,f_{i-1}\rangle+I^2} \longrightarrow \dfrac{A}{\langle f_1, \ldots,f_{i}\rangle+I^2} \longrightarrow 0.
\end{equation}

Since $\{ f_1, \ldots, f_n \}$ forms a $d$-sequence with $\{ f_1, \ldots, f_{n-1} \}$ forming a regular sequence, from Note \ref{properties}(\ref{p12}), we have
$\reg \left( \dfrac{A}{(\langle f_1, \ldots, f_{i-1} \rangle + I^2):f_i} \right)(-d_i) = \reg \left( \dfrac{A}{(\langle f_1, \ldots, f_{i-1} \rangle:f_i) + I} \right)(-d_i).$ From Note \ref{properties}(\ref{p2}), this is equal to $\reg \left( \dfrac{A}{I} \right) + d_i < \reg \left( \dfrac{A}{J} \right) + d_i \leq\reg \left( \dfrac{A}{J} \right) + d_n \leq \reg \left( \dfrac{A}{\langle f_1, \ldots, f_{i} \rangle + I^2} \right)+1$ (induction hypothesis). Hence, from Lemma \ref{reg_lemma}(\ref{r1}), 

$$\reg \left(\dfrac{A}{\langle f_1, \ldots, f_{i-1} \rangle + I^2} \right) = \reg \left( \dfrac{A}{\langle f_1, \ldots, f_{i} \rangle + I^2} \right) = \reg \left( \dfrac{A}{J} \right) +d_n-1.$$
\item Assume that $\reg(A/I) > \reg(A/J)$. Then, following the idea as in the previous case, one observes the following:
 \begin{equation} \label{eq2.1}
\reg \left( \dfrac{A}{J:f_{n}} \right) = \reg \left( \dfrac{A}{I} \right) - d_n+1
\end{equation}
and 
\begin{equation} \label{}
\reg \left( \dfrac{A}{J+\langle f_{n}^2 \rangle} \right)=\reg \left( \dfrac{A}{I} \right) +d_n.
\end{equation}

The above equations of regularity of modules along with the short exact sequence (\ref{EQ1}) give the required equality. 
\end{enumerate}
\end{proof}

\begin{lem} \label{lem2}
Consider the notation and assumptions as in Setup \ref{s1}. Then, for $s \geq 2$ and $i=0, \ldots,n-1$, 
\begin{enumerate}[\rm (i)]
\item If $\reg(A/I) < \reg(A/J)$, then,
$$ \reg \left( \dfrac{A}{(J:f_n)+\langle f_n^s \rangle} \right)=\reg(A/J)+ (s-1)d_n - 1.$$
\item If $\reg(A/I) > \reg(A/J)$, then,
$$ \reg \left( \dfrac{A}{(J:f_n)+\langle f_n^s \rangle} \right)=\reg(A/I)+ (s-1)d_n.$$
\end{enumerate}
\end{lem}

\begin{proof}

Consider the following short exact sequence
 $$0 \longrightarrow \dfrac{A}{J:f_n^{s+1}}(-sd_n) \stackrel {.f_{n}^s}
    \longrightarrow \dfrac{A}{J:f_n} \longrightarrow \dfrac{A}{(J:f_n)+\langle f_n^s \rangle} \longrightarrow 0.$$

First we claim that,
$$ \reg \left( \dfrac{A}{(J:f_n)+\langle f_n^s \rangle} \right)=\reg \left( \dfrac{A}{J:f_n} \right)+sd_n-1.$$

From Note \ref{properties}(\ref{p11}), one obtains $\reg \left( \dfrac{A}{J:f_n^{s+1}} \right) (-sd_n)=\reg \left( \dfrac{A}{J:f_n} \right)+sd_n$. Then from Lemma \ref{reg_lemma}(\ref{r3}), 
\begin{equation} \label{eq3}
\reg \left( \dfrac{A}{(J:f_n)+\langle f_n^s \rangle} \right)=\reg \left( \dfrac{A}{J:f_n} \right)+sd_n-1.
\end{equation}

Now substituting the value of $\reg \left( \dfrac{A}{J:f_n} \right)$ from Equation (\ref{eq2}) (if $\reg(A/I) < \reg(A/J)$ ) or Equation (\ref{eq2.1}) (if $\reg(A/I) > \reg(A/J)$ ) in Equation (\ref{eq3}), we get the required results.
\end{proof}

The following lemma, whose proof is immediate, will be used in this section to imply an important isomophism of modules.

\begin{lem} \label{Genlem}
Let $A$ be a commutative ring, and let $M$, $N$, and $P$ be $A$-modules such that $N$ is a submodule of $M$. Define
$$
\psi : \frac{M}{N} \longrightarrow \frac{M+P}{N+P}, 
\qquad \psi(\alpha + N) = \alpha + N + P \quad \text{for all } \alpha \in M.
$$
Then $\psi$ is a surjective $A$-module homomorphism, and its kernel is given by $\ker(\psi) = \dfrac{N + (M \cap P)}{N}$.
\end{lem}
We now proceed to prove one of the main results of this section.

\begin{thm} \label{powers_aci}
   Let $A$ be a standard graded polynomial ring over an infinite field $K$ and $I$ be a homogeneous ideal of $A$ generated by a $d$-sequence $\{ f_1, \ldots,f_{n}\}$ in degrees $d_1, \ldots, d_n$, where $\{f_1, \ldots,f_{n-2}\}$ is a regular sequence and $J=\langle f_1, \ldots,f_{n-1} \rangle$. Assume that $d_n$ is the maximum among the degrees $d_i, i=1, \ldots,n$. 
  Then for all $s \geq 2$ and $i=0,\ldots, n-1$, the following is true.
\begin{enumerate}[\rm (i)]
\item If $\reg(A/I) < \reg(A/J)$ and $\reg \left( \dfrac{(\langle f_1, \ldots, f_{n-2} \rangle:f_{n-1})}{\langle f_1, \ldots,f_{n-2} \rangle} \right)<\reg \left( \dfrac{A}{J} \right)+d_n-1$, then
$$\reg \left( \dfrac{A}{\langle f_1,\ldots,f_{i} \rangle+ I^s}  \right)=\reg(A/J)+ (s-1)d_n - 1.$$
\item If $\reg(A/I) > \reg(A/J)$ and $\reg \left( \dfrac{(\langle f_1, \ldots, f_{n-2} \rangle:f_{n-1})}{\langle f_1, \ldots,f_{n-2} \rangle} \right)<\reg \left( \dfrac{A}{I} \right)+d_n$, then
$$ \reg \left( \dfrac{A}{\langle f_1,\ldots,f_{i} \rangle+ I^s}  \right)=\reg(A/I)+ (s-1)d_n.$$
\end{enumerate}
\end{thm}

\begin{proof}
We prove by induction on $s \geq 2$. The base case for $s=2$ is a consequence of Lemma \ref{lem1}. Let $s \geq 3$.
For all $i,s \in \NN $ with $i<n$, we look at the following short exact sequence
\begin{equation} \label{eq1}
0 \longrightarrow \dfrac{A}{(\langle f_1, \ldots,f_{i}\rangle+I^s):f_{i+1}}(-d_{i+1}) \stackrel {.f_{i+1}}
    \longrightarrow \dfrac{A}{\langle f_1, \ldots,f_{i}\rangle+I^s} \longrightarrow \dfrac{A}{\langle f_1, \ldots,f_{i+1}\rangle+I^s} \longrightarrow 0.
\end{equation}

    We prove the result by descending induction on $i$. Consider the base case $i=n-1$ in Equation (\ref{eq1}). Since, $\langle f_1, \ldots,f_{(n-1)+1}\rangle+I^s=I$, from Note \ref{properties}(\ref{p12}) and Lemma \ref{lem2}, $$\reg \left( \dfrac{A}{(J+I^s):f_n}(-d_n) \right) =\reg \left( \dfrac{A}{(J:f_n)+I^{s-1}}\right)+d_n =\reg \left( \dfrac{A}{(J:f_n)+\langle f_n^{s-1} \rangle}\right)+d_n.$$

  \noindent{\textbf{Proof of (i)}:} Since $\reg(A/I) < \reg(A/J)$, $\reg \left( \dfrac{A}{(J:f_n)+\langle f_n^{s-1} \rangle}\right)+d_n=\reg(A/J)+ (s-2)d_n - 1+d_n  > \reg \left(  \dfrac{A}{I}\right)+1$. Hence $\reg \left(\dfrac{A}{J+I^s}\right)=\reg(A/J)+ (s-2)d_n - 1$ by Lemma \ref{reg_lemma}(\ref{r1}). Now assume $i \leq n-2$ and that the result holds for $i+1$.  In the short exact sequence (\ref{eq1}), by induction hypothesis on $i$, $$\reg \left(\dfrac{A}{\langle f_1, \ldots,f_{i+1}\rangle+I^s}\right)=\reg(A/J)+ (s-1)d_n - 1.$$ Since $\{f_1, \ldots,f_i \}$ forms a $d$-sequence for $1 \leq i \leq n$, from Note \ref{properties}(\ref{p12}),  $$\reg \left(\dfrac{A}{(\langle f_1, \ldots,f_{i}\rangle+I^s):f_{i+1}}\right)=\reg \left(\dfrac{A}{(\langle f_1, \ldots,f_{i}\rangle : f_{i+1})+I^{s-1}} \right).$$ Consider the following short exact sequence,
\begin{equation} \label{ses_inequality_reg}
0 \longrightarrow \langle f_1, \ldots,f_{n-2} \rangle+I^{s-1} \xhookrightarrow{}
     (\langle f_1, \ldots,f_{n-2} \rangle : f_{n-1})+I^{s-1} \twoheadrightarrow \dfrac{(\langle f_1, \ldots,f_{n-2} \rangle : f_{n-1})+I^{s-1}}{\langle f_1, \ldots,f_{n-2} \rangle+I^{s-1}} \longrightarrow 0.
\end{equation}
Now, since $\{f_1, \ldots, f_{n-1}\}$ is a $d$-sequence,  
$$(\langle f_1, \ldots,f_{n-2} \rangle : f_{n-1}) \cap I^{s-1} \subseteq (\langle f_1, \ldots,f_{n-2} \rangle : f_{n-1}) \cap I = \langle f_1, \ldots,f_{n-2} \rangle.$$ Thus applying Lemma \ref{Genlem} with $M=(\langle f_1, \ldots,f_{n-2} \rangle : f_{n-1})+I^{s-1}$, $N=\langle f_1, \ldots,f_{n-2} \rangle$ and $P=I^{s-1}$, one obtains the following isomorphism, $$\dfrac{(\langle f_1, \ldots,f_{n-2} \rangle : f_{n-1})+I^{s-1}}{\langle f_1, \ldots,f_{n-2} \rangle+I^{s-1}} 
 \cong \dfrac{(\langle f_1, \ldots,f_{n-2} \rangle : f_{n-1})}{\langle f_1, \ldots,f_{n-2} \rangle}.$$ Since $\reg \left( \dfrac{(\langle f_1, \ldots, f_{n-2} \rangle:f_{n-1})}{\langle f_1, \ldots,f_{n-2} \rangle} \right)<\reg \left( \dfrac{A}{J} \right)+d_n-1$, this implies $$\reg \left(\dfrac{A}{(\langle f_1, \ldots,f_{n-2}\rangle : f_{n-1})+I^{s-1}} \right)=
\reg \left( \dfrac{A}{\langle f_1, \ldots,f_{n-2} \rangle+I^{s-1}} \right).$$ Thus we obtain, $$\reg \left(\dfrac{A}{(\langle f_1, \ldots,f_{n-2}\rangle : f_{n-1})+I^{s-1}} \right)+d_{n-1}=\reg(A/J)+ (s-2)d_n - 1+d_{n-1} < \reg(A/J)+ (s-1)d_n - 1+1.$$ Then, by Lemma \ref{reg_lemma}(\ref{r1}), one obtains the required result for $i=n-2$ and $s > 2$. The remaining cases of $i<n-2$ can be seen as consequence of the fact that $f_1, \ldots, f_{n-2}$ forms a regular sequence and hence $$\reg \left(\dfrac{A}{(\langle f_1, \ldots,f_{i}\rangle+I^s):f_{i+1}}\right)=\reg \left(\dfrac{A}{\langle f_1, \ldots,f_{i}\rangle+I^{s-1}} \right)=\reg(A/J)+ (s-1)d_n - 1$$.

\vspace{1mm}

The case (ii) follows by arguments analogous to the first.
\end{proof}

It is important to note that the above theorem does not address the case when $\reg(A/I) = \reg(A/J)$. In the following subsection, we examine one such class of ideals, known as the Huneke–Ulrich ideals.

\subsection{Regularity of powers of Huneke-Ulrich Ideals} \label{HU Ideals} 
In this subsection, we study the regularity of powers of Huneke-Ulrich ideals which are a family of Gorenstein ideals of deviation two. 

\vspace{2mm}

Let  
$
X=\begin{bmatrix}
0            & x_{12}     & \ldots & x_{1\,n-1}\\
-x_{12}      & 0          & \ldots & x_{2\,n-1}\\
\vdots       & \vdots     & \ddots & \vdots\\
-x_{1\,n-1}  & -x_{2\,n-1}& \ldots & 0
\end{bmatrix}
$
be a generic skew-symmetric matrix of even order $n-1$, and let  
$
Y=\begin{bmatrix}
y_1 & y_2 & \ldots & y_{n-1}
\end{bmatrix}
$
be a generic $1 \times (n-1)$ matrix. Denote by $\pf(X)$ the Pfaffian of $X$ of order $n-1$, defined as the square root of the determinant of $X$, and let $I_1(YX)$ be the ideal generated by the entries of the product $YX$. Set  
\[
J_{(n)} = I_1(YX) + \langle \pf(X) \rangle \subseteq B = K[X,Y],
\]
where $K$ is a field. Then the ideal $J_{(n)}$ is called the \emph{Huneke--Ulrich ideal} \cite{Huneke&Ulrich} corresponding to the matrix $X$. The following are some notable properties of the generators of this class of ideals.

\begin{enumerate}[(a)]
\item The ideal $J_{(n)}$ is a Gorenstein perfect prime ideal of grade $n-2$ minimally generated by $n$ elements \cite[Proposition 5.9]{Huneke&Ulrich}. 
\item The ideal $I_1(YX)=\{f_1, \ldots,f_{n-2}, f_{n-1}\}$, is a perfect prime ideal and an almost complete intersection of grade $n-2$ in $B$ \cite[Proposition 5.8]{Huneke&Ulrich}.
\item The sequence $\{f_1, \ldots, f_{n-2}\}$ forms a regular sequence in $B$ \cite[Pg 266]{Kustin}.
\end{enumerate}

\begin{prop} \label{HU_prop}
Let $n \geq 5$ be an odd integer and $J_{(n)}=\langle f_1, \ldots, f_{n-1},f_{n}\rangle$ be the Huneke-Ulrich ideal corresponding to a generic skew symmetric matrix $X$ of order $(n-1) \times (n-1)$ and a generic matrix $Y$ of order $1 \times (n-1)$. Then the generators have the following properties:
\begin{enumerate}[\rm (i)]
    \item \label{hu1.} $(\langle f_1, \ldots, f_{n-1} \rangle:f_{n})=\langle y_{1}, \ldots, y_{n-1} \rangle $.
    \item \label{hu2.} $(\langle f_1, \ldots, f_{n-2} \rangle:f_{n-1})=\langle y_{n-1}, f_1 \ldots, f_{n-2} \rangle.$
    \item \label{hu3.} $(\langle f_1, \ldots, f_{n-2} \rangle:y_{n-1})=\langle f_1 \ldots, f_{n} \rangle.$
    \item \label{hu4.} $\{f_1, \ldots, f_{n-1}\}$ forms a $d$-sequence.
    \item \label{hu5.} $\{f_1, \ldots,f_n \}$ forms a $d$-sequence.
\end{enumerate}
\end{prop}

\begin{proof}
Let $J'_{(n)}=\langle f_1, \ldots, f_{n-1} \rangle$ and $J''_{(n)}=\langle f_1, \ldots, f_{n-2} \rangle$.
\begin{enumerate}[(i)]
    \item Consider the following expression for $\pf(X)$,
    $\pf(X)=\sum_{i=2}^nx_{1i}\alpha_i.$ Then, one can observe that for each $1 \leq k \leq n-1$, $y_{k}f_n=\sum_{\substack{i=1 \\ i \neq k}}\alpha_if_i.$ Thus, we obtain $\langle y_{1}, \ldots, y_{n-1} \rangle \subseteq (J'_{(n)}:f_{n}).$ The other inclusion follows from the fact that $J'_{(n)} \subseteq \langle y_{1}, \ldots, y_{n-1} \rangle$ and that $\langle y_{1}, \ldots, y_{n-1} \rangle$ is a prime ideal in $B$.
    \item From the form of the generators $f_i$ of $J'_{(n)}$, it is not difficult to observe that $y_{n-1}f_{n-1}=\sum_{i=1}^{n-2}y_if_i$. Thus, $J''_{(n)} \subseteq \langle y_{n-1}, f_1, \ldots, f_{n-2} \rangle \subseteq (J''_{(n)}:f_{n-1})$. Now, since $\overline{y_{n-1}}$ is a prime element in $B/J''_{(n)}$, it follows that $\langle y_{n-1}, f_1, \ldots, f_{n-2} \rangle$ is a prime ideal in $B$. Hence, we obtain the other inclusion.
    \item From the proof of (\ref{hu1.}), we have $y_{n-1}f_n \in J'_{(n)}$, and from $(\ref{hu2.})$,  $y_{n-1}f_{n-1} \in J'_{(n)}$. Thus, $J_{(n)} \subseteq (J''_{(n)}:y_{n-1})$. The equality then follows from the primality of $J_{(n)}$. 
    \item Since $\{ f_1, \ldots, f_{n-2} \}$ forms a $B$-regular sequence, we have $(0:f_1) \cap J'_{(n)}=0$, and for $i=1, \ldots, n-3$, $(\langle f_1, \ldots, f_{i}\rangle:f_{i+1}) \cap J'_{(n)}=\langle f_1, \ldots, f_{i}\rangle$. Finally, from part $(\ref{hu2.})$, we will have $(J''_{(n)}:f_{n-1}) \cap J'_{(n-1)}=J''_{(n)}$ and hence $\{f_1, \ldots, f_{n-1}\}$ forms a $d$-sequence.
    \item Using the fact that $\{ f_1, \ldots, f_{n-2} \}$ forms a $B$-regular sequence and part $(\ref{hu4.})$, we obtain $(0:f_1) \cap J_{(n)}=0$, and for $i=1, \ldots, n-2$, $(\langle f_1, \ldots, f_{i}\rangle:f_{i+1}) \cap J_{(n)}=\langle f_1, \ldots, f_{i}\rangle$. Finally, from part $(\ref{hu1.})$, we can conclude $(J'_{(n)}:f_{n-1}) \cap J_{(n)}=J'_{(n)}$, and hence $\{f_1, \ldots, f_{n}\}$ forms a $d$-sequence.
\end{enumerate}
\end{proof}

\begin{thm} \label{Reg_powers_HU}
Let $I=\langle f_1, \ldots, f_n \rangle$ be a Huneke-Ulrich ideal in the polynomial ring $B=K[X,Y]$. Then, for $s \geq 2$, $i=0,\ldots, n-1$,
$$sd_n -1 \leq \reg \left( \dfrac{B}{\langle f_1,\ldots,f_{i} \rangle+ I^s} \right)  \leq sd_n.$$
\end{thm}

\begin{proof}
Set $J=\langle f_1, \ldots, f_{n-1} \rangle$ and consider the short exact sequence,

$$0 \longrightarrow \dfrac{B}{J:f_{n}}(-d_n) \stackrel {.f_{n}}
    \longrightarrow \dfrac{B}{J} \longrightarrow \dfrac{B}{I} \longrightarrow 0.$$
Then, from Lemma \ref{reg_lemma} and Proposition \ref{HU_prop}(\ref{hu1.}), $\reg(B/(J:f_n))=0.$

\noindent \textbf{Claim 1:} For $i=0,\ldots, n-1$, $$\reg \left( \dfrac{B}{\langle f_1, \ldots,f_i\rangle+I^2} \right) \leq 2d_n.$$ We prove this by descending induction on $i$. For the base case $i=n-1$, we make use of the following short exact sequence
$$0 \longrightarrow \dfrac{B}{(J:f_{n}^2)}(-2d_n) \stackrel {.f_{n}^2}
    \longrightarrow \dfrac{B}{J} \longrightarrow \dfrac{B}{J+\langle f_{n}^2 \rangle} \longrightarrow 0.$$ As a consequence of Note \ref{properties}(\ref{p11}), we have that $\reg \left( \dfrac{B}{J:\langle f_{n}^2 \rangle} \right)=\reg \left( \dfrac{B}{(J: f_{n})} \right)=0.$ Thus we obtain the equality, $\reg \left( \dfrac{B}{J:\langle f_{n}^2 \rangle}(-2d_n) \right)=2d_n$. Then from Lemma \ref{reg_lemma}(\ref{r3}), $ \reg \left( \dfrac{B}{J+\langle f_{n}^2 \rangle} \right)=2d_n-1$. 
    Now, assume $i \leq n-2$. Consider the following short exact sequence,
$$0 \longrightarrow \dfrac{B}{(\langle f_1, \ldots,f_{i-1}\rangle+I^2):f_i}(-d_i) \stackrel {.f_{i}}
    \longrightarrow \dfrac{B}{\langle f_1, \ldots,f_{i-1}\rangle+I^2} \longrightarrow \dfrac{B}{\langle f_1, \ldots,f_{i}\rangle+I^2} \longrightarrow 0.$$ Then, $\reg \left(  \dfrac{B}{(\langle f_1, \ldots,f_{i-1}\rangle+I^2):f_i}(-d_i) \right)= \reg \left( \dfrac{B}{I}\right) +d_i=2d_n-2+d_i=2d_n$ and so from Lemma \ref{reg_lemma}(\ref{r1}), one obtains $\reg \left( \dfrac{B}{\langle f_1, \ldots,f_i\rangle+I^2} \right) \leq 2d_n$.\\

\noindent \textbf{Claim 2:} For $s \geq 2$, 
$$\reg \left( \dfrac{B}{(J:f_n)+\langle f_n^s \rangle} \right)=sd_n-1.$$ To see this, consider the following short exact sequence,
$$0 \longrightarrow \dfrac{B}{J:f_n^{s+1}}(-sd_n) \stackrel {.f_{n}^s}
    \longrightarrow \dfrac{B}{J:f_n} \longrightarrow \dfrac{B}{(J:f_n)+\langle f_n^s \rangle} \longrightarrow 0.$$ Then, $\reg \left(  \dfrac{B}{(J:f_n)+\langle f_n^s \rangle} \right)=\reg \left( \dfrac{B}{(J:f_n)} \right)+sd_n-1=sd_n-1.$ Now we move on to prove the main result with the help of the claims proved before.
The base case of $s=2$ is true by \textbf{Claim 1}. For $s \geq 3$, consider the following short exact sequence
\begin{equation} \label{EQ}
0 \longrightarrow \dfrac{B}{(\langle f_1, \ldots,f_{i}\rangle+I^s):f_{i+1}}(-d_{i+1}) \stackrel {.f_{i+1}}
    \longrightarrow \dfrac{B}{\langle f_1, \ldots,f_{i}\rangle+I^s} \longrightarrow \dfrac{B}{\langle f_1, \ldots,f_{i+1}\rangle+I^s} \longrightarrow 0.
\end{equation}
We prove the result by descending induction on $i$. The base step of the descending induction in Equation (\ref{eq1}) corresponds to the case $i=n-1$. In this case, clearly, $\langle f_1, \ldots,f_{(n-1)+1}\rangle+I^s=I$. Hence, $$\reg \left( \dfrac{B}{(\langle f_1, \ldots,f_{n-1}\rangle+I^s):f_{n}}(-d_{n}) \right)=\reg \left( \dfrac{B}{(J:f_n)+\langle f_n^{s-1} \rangle} \right)=(s-1)d_n-1+d_n=sd_n-1.$$ Clearly, $\reg \left( \dfrac{B}{\langle f_1, \ldots,f_{n}\rangle+I^s}\right)=sd_n-1$. Hence from Lemma \ref{reg_lemma}(\ref{r1}), $\reg \left( \dfrac{B}{J+I^s}\right)=sd_n-1$. Now, let $i = n-2$. We compare, $\reg \left( \dfrac{B}{(\langle f_1, \ldots,f_{n-1}\rangle:f_n)+I^{s-1}} \right)$ with $\reg \left( \dfrac{B}{\langle f_1, \ldots,f_{n-2}\rangle+I^{s-1}}\right)$. 

Consider the short exact sequence,
\begin{equation} \label{ses_inequality_reg}
0 \longrightarrow \langle f_1, \ldots,f_{n-2} \rangle+I^{s-1} \xhookrightarrow{}
     (\langle f_1, \ldots,f_{n-2} \rangle : f_{n-1})+I^{s-1} \twoheadrightarrow \dfrac{(\langle f_1, \ldots,f_{n-2} \rangle : f_{n-1})+I^{s-1}}{\langle f_1, \ldots,f_{n-2} \rangle+I^{s-1}} \longrightarrow 0.
\end{equation}

Now, since $\{f_1, \ldots, f_{n-1}\}$ is a $d$-sequence,  
$$(\langle f_1, \ldots,f_{n-2} \rangle : f_{n-1}) \cap I^{s-1} \subseteq (\langle f_1, \ldots,f_{n-2} \rangle : f_{n-1}) \cap I = \langle f_1, \ldots,f_{n-2} \rangle.$$

This implies using Lemma \ref{Genlem}, one obtains, 
 $$\dfrac{(\langle f_1, \ldots,f_{n-2} \rangle : f_{n-1})+I^{s-1}}{\langle f_1, \ldots,f_{n-2} \rangle+I^{s-1}} \cong 
 \dfrac{(\langle f_1, \ldots,f_{n-2} \rangle : f_{n-1})}{\langle f_1, \ldots,f_{n-2} \rangle} \cong \dfrac{\langle y_{n-1} \rangle}{\langle f_1, \ldots,f_{n-2} \rangle}.$$ 

Since $\reg \left( \dfrac{(\langle f_1, \ldots, f_{n-2} \rangle:f_{n-1})}{\langle f_1, \ldots,f_{n-2} \rangle} \right)=0$, this implies from Lemma \ref{reg_lemma}(\ref{r1}), $$\reg \left(\dfrac{B}{(\langle f_1, \ldots,f_{n-2}\rangle : f_{n-1})+I^{s-1}} \right)=\reg \left( \dfrac{B}{\langle f_1, \ldots,f_{n-2} \rangle+I^{s-1}} \right)\leq (s-1)d_n.$$ Thus in Equation (\ref{eq1}), one obtains $\reg \left( \dfrac{B}
{\langle f_1, \ldots,f_{n-2} \rangle+I^{s}} \right)\leq sd_n.$

Now, consider $i<n-2$. Then the result follows from the regular sequence property of $\{f_1, \ldots, f_{n-2} \}$.
\end{proof}

\section{On the Structure and Properties of Certain Classes of Deviation Two Ideals} \label{Section3}

The aim of this section is to develop the homological tools required for the results of the last section. Throughout, we assume that $I$ is an ideal of deviation two in a  polynomial ring $B$. We construct two distinct classes of deviation two ideals by adjoining, in an iterative manner, a $B/I$-regular sequence to $I$. Moreover, we show that the free resolutions of the second Koszul homology of these ideals can be described explicitly in terms of the free resolution of the second Koszul homology of $I$ (see Lemma \ref{Res_of_H2} and Proposition \ref{H1-regular}).

\vspace{2mm}

It is natural to ask how one can obtain classes of ideals of deviation two beginning with a given ideal $I$. A standard approach towards this problem is to iteratively extend the ideal $I$ by adjoining elements satisfying certain properties. The following series of lemmas provides conditions under which such extensions preserve the deviation.

\begin{lem}\label{reg-grade}
Let $I$ be an ideal in a polynomial ring $B$, and let $f \in B$ such that $f$ is $B/I$-regular. 
\begin{enumerate}[(i)]
\item Then, 
$\grade(I + \langle f \rangle) = \grade(I) + 1.
$
\item \label{reg-dev} The deviation of $I + \langle f \rangle$ is equal to the deviation of $I$.
\end{enumerate}
\end{lem}

\begin{proof}
Since $B$ is Cohen--Macaulay and $I$ is an ideal of $B$, we have $\grade(I) = \height(I)$.
\begin{enumerate}[(i)]
\item By the definition of height, there exists a minimal prime $\mathfrak{p} \supseteq I$ such that $\height(\mathfrak{p}) = \height(I)$. Given that $f$ is $B/I$-regular, it follows that $f \notin \mathfrak{p}$. By Krull's Principal Ideal Theorem, there exists a prime ideal $\mathfrak{p}_1 \supseteq \mathfrak{p} + \langle f \rangle$ such that
$
\height(\mathfrak{p}_1) = \height(\mathfrak{p}) + 1 = \height(I) + 1.
$
Since $I + \langle f \rangle \subseteq \mathfrak{p}_1$, we get
$
\height(I + \langle f \rangle) \leq \height(\mathfrak{p}_1) = \height(I) + 1.
$
On the other hand, as $f$ is $B/I$-regular, it is not a zero divisor on $B/I$, so
$
\height(I + \langle f \rangle) > \height(I).
$
Hence,
$
\height(I + \langle f \rangle) \geq \height(I) + 1.
$
Combining the two inequalities, we conclude
$ \height(I + \langle f \rangle) = \height(I) + 1.$
Since $B$ is Cohen-Macaulay, we also have
$
\grade(I + \langle f \rangle) = \height(I + \langle f \rangle) = \height(I) + 1 = \grade(I) + 1. $
\item Since $f$ is $B/I$-regular, it follows that $\mu(I + \langle f \rangle) = \mu(I) + 1$. Since $\grade(I + \langle f \rangle) = \grade(I) + 1$, it follows that
\[
\mu(I + \langle f \rangle) - \grade(I + \langle f \rangle) = (\mu(I) + 1) - (\grade(I) + 1) = \mu(I) - \grade(I),
\]
which shows that the deviation remains unchanged.
\end{enumerate}
\end{proof}

\begin{lem}
Let $I$ be an unmixed ideal in a polynomial ring $B$, and let $f \in B$. Then the deviation of $I + \langle f \rangle$ is equal to the deviation of $I$ if and only if $f$ is $B/I$-regular.
\end{lem}

\begin{proof}
The forward direction follows from Lemma \ref{reg-grade} (\ref{reg-dev}).

For the converse, suppose the deviation of $I + \langle f \rangle$ is equal to the deviation of $I$. Then
\[
\mu(I + \langle f \rangle) = \mu(I) + 1 \quad \text{and} \quad \grade(I + \langle f \rangle) = \grade(I) + 1.
\]
Thus, $\height(I + \langle f \rangle) = \height(I) + 1$. Since $I$ is unmixed, we have $\ass(I) = \Min(I)$ and all associated primes of $I$ have the same height. Therefore, $f$ does not belong to any associated prime of $B/I$, and hence $f$ is $B/I$-regular.
\end{proof}

\begin{lem} \label{Reg_seq}
Let $I = \langle f_1, f_2, \ldots, f_n, g_1, g_2 \rangle$ be an ideal in a polynomial ring $B$, where $\{f_1, \ldots, f_n\}$ is a regular sequence in $B$, and $\grade(I) = n$. Let $I'=\langle f_1, \ldots, f_n \rangle$, $g \in B$ be such that the deviation of $I + \langle g \rangle$ is equal to the deviation of $I$. Then $g$ is $B/I'$-regular.
\end{lem}

\begin{proof}
Since $\{f_1, \ldots, f_n\}$ is a regular sequence in $B$, the quotient $B/I'$ is Cohen--Macaulay. By assumption, $\grade(I) = n = \grade(I')$, so $\grade(I/I') = 0$ in $B/I'$. 

Given that the deviation remains unchanged upon adding $g$, we have
\[
\grade((I + \langle g \rangle)/I') = \grade(I + \langle g \rangle) - \grade(I') = (n + 1) - n = 1.
\]
Suppose, $g$ is not $B/I'$-regular. Then there exists an associated prime $\mathfrak{p}/I'$ of $B/I'$ such that $(I + \langle g \rangle)/I' \subseteq \mathfrak{p}/I'$. This implies $\grade((I + \langle g \rangle)/I') = 0$, contradicting the previous conclusion that it equals $1$. Therefore, $g$ is $B/I'$-regular.
\end{proof}

Now, let $I$ be an ideal of deviation two in a polynomial ring $B$ and $(\mathfrak{L}_\bullet, \tau_\bullet)$ denote a graded free resolution of the second Koszul homology $H_2(I,B)$. Then, we observe the following.

\begin{lem} \label{Res_of_H2}
Let $f$ be an element in $B$ such that $f$ is $B/I$-regular and $f$ is $H_1(I,B)$-regular. Then for $J=I+\langle f \rangle$, the following complex,  

$$ \cdots \longrightarrow \mathfrak{L}'_2 \stackrel{\tau'_2} \longrightarrow \mathfrak{L}'_1 \stackrel{\tau'_1} \longrightarrow \mathfrak{L}'_0 \longrightarrow 0 $$
is a graded free resolution of $H_2(J,B)$,
where for $i \geq 1$,

$\mathfrak{L}'_i= \mathfrak{L}_{i} \oplus \mathfrak{L}_{i-1} $ with \[
\tau_{i+1}' = \begin{pmatrix}
- \tau_{i+1} & f \cdot \mathrm{id}_{\mathfrak{L}_{i}} \\
0 & \tau_{i}
\end{pmatrix},
\]

and $\mathfrak{L}'_0=\mathfrak{L}_0$,
with
\[
\tau_1' = \begin{pmatrix}
-{\tau_1} &
{f \cdot \mathrm{id}_{\mathfrak{L}_0}}\\
\end{pmatrix}.
\]
\end{lem}

\begin{proof}
Consider the natural short exact sequence of Koszul complexes,

\begin{align}
	0 \longrightarrow \mathcal{K}_\bullet(I, B) \xrightarrow{i} \mathcal{K}_\bullet(J, B) \xrightarrow{\pi} \mathcal{K}_\bullet(I, B)[-1] \longrightarrow 0.
\end{align}
Since deviation of $I$ is same as the deviation of $J$ which is equal to two, one has $H_i(I,B)=0=H_i(J,B)$ for all $i \geq 3$. Thus from this short exact sequence, we obtain a long exact sequence in Koszul homology,

\begin{align*}
	0 \longrightarrow H_2(I, B) \xrightarrow{\color{violet}\cdot {f}} H_2(I, B) \longrightarrow H_2(J, B) \longrightarrow H_1(I, B) \xrightarrow{\color{red}\cdot {f}} H_1(I, B) \longrightarrow  \\
    H_1(J, B) \longrightarrow H_0(I, B) \xrightarrow{\color{blue}\cdot {f}} H_0(I, B) \longrightarrow H_0(J, B) \longrightarrow 0.
\end{align*}

Since, $f$ is $H_1(I,B)$-regular, one obtains the following short exact sequence from the above long exact sequence of Koszul homology modules,
\begin{align*}
0 \longrightarrow H_2(I, B) \xrightarrow{\color{violet}\cdot {f}} H_2(I, B) \longrightarrow H_2(J, B) \longrightarrow 0
\end{align*}

This implies that
$$
H_2(J, B) \cong H_2(I, B)/ f  H_2(I, B),
$$
and the mapping cone of the multiplication map by $ \color{violet}f $ on $H_2(I,B)$, extended to a map of complexes, yields the required resolution of $ H_2(J, B) $ in terms of a resolution of $ H_2(I, B) $.

\end{proof}

This naturally leads to the following question regarding the behaviour of elements with respect to Koszul homology.

\begin{quest}
Given an ideal $I$ of deviation two in a polynomial ring $B$, and an element $f \in B$ that is $B/I$-regular, under what conditions is $f$ also regular on $H_1(I,B)$?  
\end{quest}

In the subsequent part of this section, we define a class of ideals that satisfy the above property, allowing for a better understanding of their second Koszul homology.

\vspace{2mm}

Let $ I_0 =\langle f_1, f_2, \ldots, f_{n-2}, f_{n-1},f_{n}\rangle$ be an ideal of deviation two in a polynomial ring $B$, where the sequence $\{f_1, f_2, \ldots, f_{n-2}\}$ forms a regular sequence in $B$. Set $I_0'=\langle f_1, f_2, \ldots, f_{n-2} \rangle$ and $I_0''=\langle f_1, f_2, \ldots, f_{n-2}, f_{n-1} 
\rangle$. Assume that $(I_0': f_{n-1}) \subseteq (I_0': f_{n})$. We define a class of ideals $\{I_k\}_{k \geq 1}$ by extending $I_0$ iteratively as follows.

\vspace{2mm}

\begin{description} \label{Classofideals}
    \item[\textbf{Class I}] \label{Class1}
 Let $I_k=I_{0}+ \langle g_1, \ldots, g_{k} \rangle$ for $k \geq 1$ where the elements $g_1, \ldots, g_k \in B$ satisfy the following properties.
\begin{enumerate}[(i)]
    \item The sequence $\{ g_1, \ldots, g_{k}\}$ is $B/I_{0}$-regular.
    \item The sequence $\{g_1, \ldots, g_k\}$ is $B/I_0''$-regular.
    \item $(I_0'+\langle g_1 ,\ldots, g_{k}\rangle: f_{n-1}) \subseteq (I_0'+\langle g_1 ,\ldots, g_{k}\rangle: f_{n}) $ for $k \geq 1$.
\end{enumerate}

\vspace{2mm}

\item [\textbf{Class II}]\label{Class2}
Let $I_k=I_{0}+ \langle g_1, \ldots, g_{k} \rangle$ for $k \geq 1$ where the elements $g_1, \ldots, g_k \in B$ satisfy the following properties. 
\begin{enumerate} [(i)]
    \item The sequence $\{g_1, \ldots, g_k\}$ is $B/I_0$-regular.
    \item The sequence $\{g_1, \ldots, g_k\}$ is $B/I_0''$-regular.
    \item The sequence \label{4} $\{g_1, \ldots, g_k\}$ is $M$-regular where $M=(I_0'':f_{n})$.
\end{enumerate}
\end{description}

\begin{prop} \label{H1-regular}
Let $\{I_k\}$ denote a class of ideals in a polynomial ring $B$ of the form in Class I or Class II. Then $\{g_1, \ldots, g_k\}$ is $H_1(I_{0},B)$ regular.
\end{prop}

\begin{proof}
 We have $I_0$ to be an ideal of deviation $2$ in $B$, and $I_0'=\langle f_1, f_2, \ldots, f_{n-2} \rangle$ and $I_0''=\langle f_1, f_2, \ldots, f_{n-2}, f_{n-1} \rangle$ to be subideals of $I_0$ of deviations $0$ and $1$, respectively. Consider the short exact sequence of Koszul complexes,
\begin{align}\label{Kos-dev-1}
	0 \longrightarrow \mathcal{K}_\bullet(I_0', B) \xrightarrow{i} \mathcal{K}_\bullet(I_0'', B) \xrightarrow{\pi} \mathcal{K}_\bullet(I_0', B)[-1] \longrightarrow 0.
\end{align}
From this short exact sequence, and using the fact that $I_0'$ and $I_0''$ are ideals of deviation $0$ and $1$ respectively, we obtain a long exact sequence in Koszul homology,
\begin{align}\label{long-kos-dev-1}
	0 \longrightarrow H_1(I_0'', B) \longrightarrow H_0(I_0', B) \xrightarrow{\color{blue}\cdot{f_{n-1}}} H_0(I_0', B)\longrightarrow H_0(I_0'', B) \longrightarrow 0.
\end{align}
Here, we use the fact that $H_1(I_0', B)[-1]) = H_0(I_0', B) \cong B/I_0'$. From the exactness of \eqref{long-kos-dev-1}, it follows that,
\begin{equation} \label{LES1}
H_1(I_0'', B) = \ker(\textcolor{blue}{\cdot{f_{n-1}}}) = (0 :_{B/I_0'} f_{n-1}).
\end{equation}

Now, consider the second short exact sequence of Koszul complexes,
\begin{align}\label{Kos-dev-1'}
	0 \longrightarrow \mathcal{K}_\bullet(I_0'', B) \xrightarrow{i} \mathcal{K}_\bullet(I_0, B) \xrightarrow{\pi} \mathcal{K}_\bullet(I_0'', B)[-1] \longrightarrow 0.
\end{align}
This induces the following long exact sequence in Koszul homology,
\begin{align*}
	0 \longrightarrow  H_2(I_0, B) \longrightarrow H_1(I_0'', B) \xrightarrow{\color{red}\cdot{f_n}} H_1(I_0'', B)  \xrightarrow{\psi} H_1(I_0, B) \xrightarrow{\epsilon}  \\ 
    H_0(I_0'', B) \xrightarrow{\color{blue}\cdot{f_n}} H_0(I_0'', B) \longrightarrow H_0(I_0, B) \longrightarrow 0.
\end{align*}

From this sequence and Equation $(\ref{LES1})$, we obtain,

$$
H_2(I_0, B) \cong \ker(\red{\cdot{f_n}}) = (0 :_{B/I_0'} f_{n-1}) \cap (0 :_{B/I_0'} f_n).
$$
Since $(I_0' : f_{n-1}) \subseteq (I_0' : f_n)$, it follows that
$$
H_2(I_0, B) = (0 :_{B/I_0'} f_{n-1}), \quad \text{and} \quad  \ker(\psi) = \operatorname{im}(\red{\cdot{f_n}}) = 0.
$$
Moreover, $\operatorname{im}(\epsilon) = \ker(\textcolor{blue}{\cdot{f_n}}) = (0 :_{B/I_0''} f_n)$, so we obtain a short exact sequence,
\begin{align} \label{ses1}
	0 \longrightarrow H_1(I_0'', B) \xrightarrow{\psi} H_1(I_0, B) \xrightarrow{\epsilon} (0 :_{B/I_0''} f_n) \longrightarrow 0.
\end{align}

Now consider the following commutative diagram of short exact sequences,
\begin{equation}\label{CD1}
\begin{tikzcd}
0 \arrow[r] & H_1(I_0'', B) \arrow[r, "\psi"] \arrow[d, "{\textcolor{violet}{\cdot g_1}}"] & H_1(I_0, B) \arrow[r, "\epsilon"] \arrow[d, "\textcolor{red}{{\cdot g_1}}"] & (0 :_{B/I_0''} f_n) \arrow[r] \arrow[d, "\textcolor{blue}{\cdot g_1}"] & 0 \\
0 \arrow[r] & H_1(I_0'', B) \arrow[r, "\psi"] & H_1(I_0, B) \arrow[r, "\epsilon"] & (0 :_{B/I_0''} f_n) \arrow[r] & 0
\end{tikzcd}
\end{equation}

By applying snake lemma, we obtain the following long exact sequence of $B$-modules,
\begin{align}
	0 \longrightarrow \ker(\textcolor{violet}{{\cdot g_1}}) \longrightarrow \ker(\textcolor{red}{{\cdot g_1}}) \longrightarrow \ker(\textcolor{blue}{{\cdot g_1}}) \longrightarrow \operatorname{coker}(\textcolor{violet}{{\cdot g_1}}) \longrightarrow \operatorname{coker}(\textcolor{red}{{\cdot g_1}}) \longrightarrow \operatorname{coker}(\textcolor{blue}{{\cdot g_1}}) \longrightarrow 0.
\end{align}

Now, observe that,
\begin{equation}
\ker(\textcolor{blue}{{\cdot g_1}}) = (0 :_{B/I_0''} f_n) \cap (0 :_{B/I_0''} g_1)
\end{equation} and
\begin{equation} \label{end}
\ker(\textcolor{violet}{{\cdot g_1}}) = (0 :_{B/I_0'} f_{n-1}) \cap (0 :_{B/I_0'} g_1).
\end{equation}
Since $g_1$ is regular on $B/I_0''$ (part of the assumptions in the definition of Class I and Class II) and $B/I_0'$ (follows from Lemma \ref{Reg_seq}), the maps $\textcolor{violet}{{\cdot g_1}}$ and $\textcolor{blue}{{\cdot g_1}}$ are injective, implying that  $\textcolor{red}{{\cdot g_1}}$ is also injective. That is, $g_1$ is a regular element on $H_1(I_0, B)$.
\vspace{2mm}

\noindent\textbf{Class I.}
Now, for $I_k= I_0 + \langle g_1, \ldots, g_k\rangle $, set $I_k'=I_0' + \langle g_1, \ldots, g_k\rangle$ and $I_k''=I_0'' + \langle g_1, \ldots, g_k\rangle$. Then by following the steps from Equation (\ref{Kos-dev-1}) to Equation (\ref{end}) inductively, and using the assumptions in the definition of Class I, one obtains that $\{g_1, \ldots, g_k \}$ is $H_1(I_{0},B)$ regular.

\noindent\textbf{Class II.}
Based on the previous discussions, the commuting diagram (\ref{CD1}) can be extended to obtain the following,
\begin{equation}\label{CD2}
\begin{tikzcd}
  & 0  \arrow[d] & 0  \arrow[d] & 0  \arrow[d] & \\
0 \arrow[r] & H_1(I_0'', B) \arrow[r, "\psi"] \arrow[d, "{\textcolor{violet}{\cdot g_1}}"] & H_1(I_0, B) \arrow[r, "\epsilon"] \arrow[d, "\textcolor{red}{{\cdot g_1}}"] & (0 :_{B/I_0''} f_n) \arrow[r] \arrow[d, "\textcolor{blue}{\cdot g_1}"] & 0 \\
0 \arrow[r] & H_1(I_0'', B) \arrow[r, "\psi"] \arrow[d] & H_1(I_0, B) \arrow[r, "\epsilon"] \arrow[d] & (0 :_{B/I_0''} f_n) \arrow[r] \arrow[d] & 0\\
0 \arrow[r] & \dfrac{H_1(I_0'', B)}{g_1 H_1(I_0'', B)} \arrow[r] \arrow[d] & \dfrac{H_1(I_0, B)}{g_1 H_1(I_0, B)} \arrow[r] \arrow[d] & \dfrac{(0 :_{B/I_0''} f_n)}{g_1 (0 :_{B/I_0''} f_n)} \arrow[r] \arrow[d] & 0 \\
& 0   & 0   & 0  & 
\end{tikzcd}
\end{equation}

Thus we obtain the following exact sequence,
\begin{align}
	0 \longrightarrow \frac{H_1(I_0'', B)}{g_1 H_1(I_0'', B)} \longrightarrow \frac{H_1(I_0, B)}{g_1 H_1(I_0, B)} \longrightarrow \frac{(0 :_{B/I_0''} f_n)}{g_1 (0 :_{B/I_0''} f_n)} \longrightarrow 0.
\end{align}
By applying arguments similar to those following Equation (\ref{ses1}) inductively, and using the assumptions in the definition of Class II, we conclude that the sequence $\{g_1, \dots, g_k\}$ is regular on $H_1(I_0, B)$. 
\end{proof}

\subsection{Families of binomial edge ideals of deviation two via combinatorial constructions}
Let $G$ be a finite simple graph on the vertex set ${x_1, \ldots, x_n}$ with edge set $E(G)$. Let $B = K[x_1, \ldots, x_n, y_1, \ldots, y_n]$ be a polynomial ring over a field $K$.  
  
The \textit{binomial edge ideal} of $G$ in $B$, denoted by $J_G$, is defined as $J_G = \langle f_{ij} = x_i y_j - x_j y_i \;\mid\; \{i,j\} \in E(G)\rangle$ \cite{Herzog}.
Some classes of ideals that fall under Class I include the binomial edge ideals corresponding to graphs of the form given below.

\vspace{2mm}

Let $\ell, m, n$ be positive integers. We define two classes of graphs $ G_{1,(\ell,m,n)} $ and $ G_{2,(\ell,m,n)} $ as follows:

\vspace{3mm}

\noindent \textbf{The graphs $ G_{1,(\ell,m,n)}$:} Let $\ell \geq 2$ and $ m,n \geq 1$. Let $ P_\ell $ be a path on $ \ell $ vertices with vertex sequence $ \alpha_1, \alpha_2, \ldots, \alpha_\ell $, where each $ \alpha_i $ is connected to $ \alpha_{i+1} $ for $ 1 \leq i < \ell $. To the vertex $ \alpha_j \in P_\ell $, where $1 < j < \ell$ (that is, $\alpha_j$ is not a pendent vertex), attach two additional paths: 
 \begin{itemize}
        \item A path $ P_m $ with vertices $ \beta_1, \beta_2, \ldots, \beta_m $, where $ \beta_1 $ is connected to $ \alpha_j $, and $ \beta_i $ is connected to $ \beta_{i+1} $ for $ 1 \leq i < m $.
        \item A path $ P_n $ with vertices $ \gamma_1, \gamma_2, \ldots, \gamma_n $, where $ \gamma_1 $ is connected to $ \alpha_j $, and $ \gamma_i $ is connected to $ \gamma_{i+1} $ for $ 1 \leq i < n $.
    \end{itemize}

Thus, the graph $ G_{1,(\ell,m,n)} $ consists of a central path $ P_\ell $, with two paths $ P_m $ and $ P_n $ attached to the same non-pendant vertex $ x_j $ of $ P_\ell $.    

\begin{description} \label{Graphs_description}
    \item[\textbf{The graphs $ G_{1,(\ell,m,n)} $}]
    
    Let $\ell \geq 2$ and $ m,n \geq 1$. Let $ P_\ell $ be a path on $ \ell $ vertices with vertex sequence $ \alpha_1, \alpha_2, \ldots, \alpha_\ell $, where each $ \alpha_i $ is connected to $ \alpha_{i+1} $ for $ 1 \leq i < \ell $. To the vertex $ \alpha_j \in P_\ell $, where $1 < j < \ell$ (that is, $\alpha_j$ is not a pendent vertex), attach two additional paths:
    \begin{itemize}
        \item A path $ P_m $ with vertices $ \beta_1, \beta_2, \ldots, \beta_m $, where $ \beta_1 $ is connected to $ \alpha_j $, and $ \beta_i $ is connected to $ \beta_{i+1} $ for $ 1 \leq i < m $.
        \item A path $ P_n $ with vertices $ \gamma_1, \gamma_2, \ldots, \gamma_n $, where $ \gamma_1 $ is connected to $ \alpha_j $, and $ \gamma_i $ is connected to $ \gamma_{i+1} $ for $ 1 \leq i < n $.
    \end{itemize}
    
    Thus, the graph $ G_{1,(\ell,m,n)} $ consists of a central path $ P_\ell $, with two paths $ P_m $ and $ P_n $ attached to the same non-pendant vertex $ \red{\alpha_j} $ of $ P_\ell $.

\vspace{3mm}

$G_{1,(\ell,m,n)}:$ 
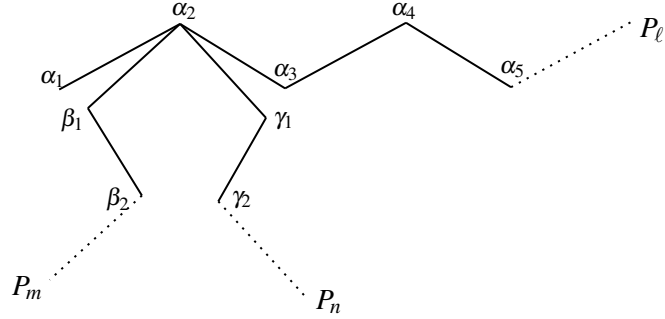
\begin{center}
\tikzset{every picture/.style={line width=0.75pt}} 

\begin{tikzpicture}[x=0.75pt,y=0.75pt,yscale=-1,xscale=1]

\draw    (120.24,118.22) -- (180.86,85.3) ;
\draw    (233.5,117.64) -- (180.86,85.3) ;
\draw    (233.5,117.64) -- (294.12,84.72) ;
\draw    (346.76,117.05) -- (294.12,84.72) ;
\draw  [dash pattern={on 0.84pt off 2.51pt}]  (346.76,117.05) -- (407.38,84.13) ;
\draw    (180.86,85.3) -- (223.93,132.44) ;
\draw    (223.93,132.44) -- (200,174.33) ;
\draw  [dash pattern={on 0.84pt off 2.51pt}]  (200,174.33) -- (243.07,221.47) ;
\draw    (180.86,85.3) -- (134.6,127.57) ;
\draw    (134.6,127.57) -- (161.71,171.4) ;
\draw  [dash pattern={on 0.84pt off 2.51pt}]  (161.71,171.4) -- (115.45,213.68) ;

\draw (410.07,80.91) node [anchor=north west][inner sep=0.75pt]   [align=left] {$P_\ell$};
\draw (95.12,210.3) node [anchor=north west][inner sep=0.75pt]   [align=left] {$P_m$};
\draw (247.49,217.27) node [anchor=north west][inner sep=0.75pt]   [align=left] {$P_n$};
\draw (108.99,107.47) node [anchor=north west][inner sep=0.75pt]   [align=left] {{\footnotesize $\alpha_1$}};
\draw (174.94,73.3) node [anchor=north west][inner sep=0.75pt]   [align=left] {{\footnotesize $\alpha_2$}};
\draw (225.99,103.96) node [anchor=north west][inner sep=0.75pt]   [align=left] {{\footnotesize $\alpha_3$}};
\draw (285.18,73.33) node [anchor=north west][inner sep=0.75pt]   [align=left] {{\footnotesize $\alpha_4$}};
\draw (339.65,102.4) node [anchor=north west][inner sep=0.75pt]   [align=left] {{\footnotesize $\alpha_5$}};
\draw (119.54,125.9) node [anchor=north west][inner sep=0.75pt]   [align=left] {{\footnotesize $\beta_1$}};
\draw (142.3,165.91) node [anchor=north west][inner sep=0.75pt]   [align=left] {{\footnotesize $\beta_2$}};
\draw (226.37,127.9) node [anchor=north west][inner sep=0.75pt]   [align=left] {{\footnotesize $\gamma_1$}};
\draw (205.23,167.86) node [anchor=north west][inner sep=0.75pt]   [align=left] {{\footnotesize $\gamma_2$}};

\end{tikzpicture}
\end{center}

\vspace{3mm}
    
    \item[\textbf{The graphs $ G_{2,(\ell,m,n)} $}]
    
    Let $\ell \geq 3$ and $ m,n \geq 1$. Let $ C_\ell $ be a cycle on $ \ell $ vertices with vertex sequence $ \alpha_1, \alpha_2, \ldots, \alpha_\ell $, where each $ \alpha_i $ is connected to $ \alpha_{i+1} $ for $ 1 \leq i < \ell $, and $ \alpha_\ell $ is connected back to $ \alpha_1 $. To a chosen vertex $ \alpha_j \in C_\ell $, where $1 \leq j \leq \ell$, attach two paths:
    \begin{itemize}
        \item A path $ P_m $ with vertices $ \beta_1, \beta_2, \ldots, \beta_m $, where $ \beta_1 $ is connected to $ \alpha_j $, and $ \beta_i $ is connected to $ \beta_{i+1} $ for $ 1 \leq i < m $.
        \item A path $ P_n $ with vertices $ \gamma_1, \gamma_2, \ldots, \gamma_n $, where $ \gamma_1 $ is connected to $ x_j $, and $ \gamma_i $ is connected to $ z_{i+1} $ for $ 1 \leq i < n $.
    \end{itemize}
    
    Thus, the graph $ G_{2,(\ell,m,n)} $ consists of a central cycle $ C_\ell $, with two paths $ P_m $ and $P_n$ attached to the same vertex $\alpha_j$ of the cycle.

\vspace{3mm}

$G_{2,(\ell,m,n)}:$ 
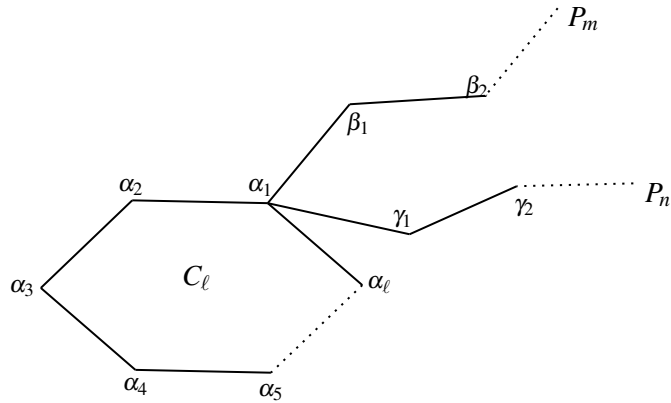
\begin{center}
\tikzset{every picture/.style={line width=0.75pt}} 

\begin{tikzpicture}[x=0.75pt,y=0.75pt,yscale=-1,xscale=1]

\draw    (228.75,150.79) -- (182.92,194.73) ;
\draw    (296.69,152.2) -- (228.75,150.79) ;
\draw    (182.92,194.73) -- (230.33,235.83) ;
\draw    (298.28,237.25) -- (230.33,235.83) ;
\draw   (296.69,152.2) -- (344.1,193.31) ;
\draw   [dash pattern={on 0.84pt off 2.51pt}] (344.1,193.31) -- (298.28,237.25) ;
\draw    (337.78,102.59) -- (296.69,152.2) ;
\draw    (405.73,98.34) -- (337.78,102.59) ;
\draw  [dash pattern={on 0.84pt off 2.51pt}]  (442.07,54.4) -- (417.58,84.01) -- (405.73,98.34) ;
\draw    (367.8,167.8) -- (296.69,152.2) ;
\draw    (421.53,143.7) -- (367.8,167.8) ;
\draw  [dash pattern={on 0.84pt off 2.51pt}]  (480,142.28) -- (421.53,143.7) ;

\draw (253.06,181.95) node [anchor=north west][inner sep=0.75pt]   [align=left] {$C_\ell$};
\draw (165.67,189.28) node [anchor=north west][inner sep=0.75pt]   [align=left] {{\footnotesize $\alpha_3$}};
\draw (290.36,241.34) node [anchor=north west][inner sep=0.75pt]   [align=left] {{\footnotesize $\alpha_5$}};
\draw (222.55,238.65) node [anchor=north west][inner sep=0.75pt]   [align=left] {{\footnotesize $\alpha_4$}};
\draw (220.13,138.67) node [anchor=north west][inner sep=0.75pt]   [align=left] {{\footnotesize $\alpha_2$}};
\draw (285.08,138.84) node [anchor=north west][inner sep=0.75pt]   [align=left] {{\footnotesize $\alpha_1$}};
\draw (345.09,187.53) node [anchor=north west][inner sep=0.75pt]   [align=left] {{\footnotesize $\alpha_\ell$}};
\draw (334.71,105.16) node [anchor=north west][inner sep=0.75pt]   [align=left] {{\footnotesize $\beta_1$}};
\draw (394.22,86.39) node [anchor=north west][inner sep=0.75pt]   [align=left] {{\footnotesize $\beta_2$}};
\draw (357.68,154.43) node [anchor=north west][inner sep=0.75pt]   [align=left] {{\footnotesize $\gamma_1$}};
\draw (418.78,148.5) node [anchor=north west][inner sep=0.75pt]   [align=left] {{\footnotesize $\gamma_2$}};

\draw (442.22,52.39) node [anchor=north west][inner sep=0.75pt]   [align=left] {{ $P_m$}};

\draw (480.78,140.5) node [anchor=north west][inner sep=0.75pt]   [align=left] {{ $P_n$}};

\end{tikzpicture}
\end{center}
\end{description}

Next, we recall certain basic concepts and results on graphs and their associated binomial edge ideals, which will be used in Proposition \ref{GRAPHS} to compute colon ideals.

\begin{notation} \label{Notation1}

Let $G$ be a simple graph. For an edge $e' \in G$, we denote by $G \setminus e'$ the graph with the same vertex set $V(G)$ but with edge set $E(G) \setminus \{e'\}$. An edge $e'$ is called a \textit{bridge} if its removal increases the number of connected components of the graph, i.e., $c(G) < c(G \setminus e')$, where $c(G)$ denotes the number of connected components of $G$.  

For a vertex $v \in V(G)$, the \textit{neighborhood} of $v$ in $G$ is defined as
\[
N_G(v) = \{ u \in V(G) \mid \{u,v\} \in E(G) \}.
\]
Let $e = \{i,j\} \notin E(G)$ be an edge that is added to $G$ to form $G \cup \{e\}$. Then the graph $G_e$ is defined on the same vertex set $V(G)$ with edge set
\[
E(G_e) = E(G) \cup \{ \{p,q\} \mid p,q \in N_G(i) \ \text{or} \ p,q \in N_G(j) \},
\]
see \cite[Definition 3.1]{MS}
\end{notation}

\begin{lem} \label{Colon}
Let $G$ be a simple graph and $J_G$ denote the corresponding binomial edge ideal. The following results describe certain colon ideals of binomial edge ideals.
\begin{enumerate}[(i)]
    \item \label{Colon1} \cite[Theorem 3.4]{MS}  
    Let $e = \{i,j\} \notin E(G)$ be a bridge in $G \cup \{e\}$. Then  
    $$
    (J_G : f_e) = J_{G_e}.
    $$
    
    \item \label{Colon2} \cite[Theorem 3.7]{MS}  
    Let $e = \{i,j\} \in E(G)$. Then
    $$
    (J_G : f_e) = J_{G_e} + \big( g_{P,t} \ \big| \ P: i,i_1,\ldots,i_s,j \ \text{is a path between $i$ and $j$ and} \ 0 \leq t \leq s \big),
    $$
    where $g_{P,0} = x_{i_1} \cdots x_{i_s}$ and for each $1 \leq t \leq s$,
    $$
    g_{P,t} = y_{i_1} \cdots y_{i_t} x_{i_{t+1}} \cdots x_{i_s}.
    $$
\end{enumerate}
\end{lem}

In the following result, we use the notations introduced in the definitions of Class I and Class II, as well as those in Notation \ref{Notation1}.

\begin{prop}\label{GRAPHS}
The binomial edge ideals corresponding to the graphs $G_{1,(\ell,m,n)}$ and $G_{2,(\ell,m,n)}$ in their respective polynomial rings satisfy the properties of \emph{Class I}.
\end{prop}

\begin{proof}
We begin with the family of graphs $G_{1,(\ell,m,n)}$. Without loss of generality, let $I_0$ denote the binomial edge ideal corresponding to the graph $G_{1,(\ell,1,1)}$. From \cite{EHH,Rinaldo}, the edge binomials of the path graph $P_\ell$ form a regular sequence. Hence $J_{P_\ell}$ is a complete intersection ideal. Furthermore, Lemma~\ref{Colon} implies that 
$$
(J_{P_{\ell}}:f_{\alpha_2 \beta_1}) \neq J_{P_{\ell}} \quad \text{and} \quad (J_{P_{\ell}}:f_{\alpha_2 \gamma_1}) \neq J_{{P_{\ell}}}.
$$
Therefore, by Lemma \ref{Reg_seq}, $I_0$ is a deviation two ideal in the corresponding polynomial ring.

Next, let $I_1$ be the binomial edge ideal of the graph $G_{1,(\ell,2,1)}$. Applying Lemma~\ref{Colon}, we deduce that $f_{\beta_1 \beta_2}$ is both $B/I_0$-regular and $B/I_0''$-regular. Consequently, by Lemma \ref{reg-grade}, the binomial edge ideal corresponding to $G_{1,(\ell,2,1)}$ is also of deviation two. Moreover, Lemma \ref{Colon} gives
$$
(I_0' + \langle f_{\beta_1 \beta_2} \rangle : f_{\alpha_2 \beta_1}) \subseteq (I_0' + \langle f_{\beta_1 \beta_2} \rangle : f_{\alpha_2 \gamma_1}),
$$

Thus, the pair of ideals $\{I_0, I_1\}$ satisfies the assumptions stated in the definition of Class I. Repeating this argument inductively, we conclude that the graphs $G_{1,(\ell,m,n)}$ yield a family of ideals satisfying the defining conditions of both Class I.

Now, we consider the family of graphs $G_{2,(\ell,m,n)}$. Without loss of generality, denote by $I_0$ the binomial edge ideal corresponding to $G_{2,(\ell,1,1)}$. From the colon conditions—obtained using Lemma \ref{Colon}, we have  
$$
(J_{P_{\ell-1}} : f_{\alpha_{\ell-1} \alpha_\ell}) \neq J_{P_{\ell-1}}, 
\quad (J_{P_{\ell-1}} : f_{\alpha_{1} \beta_1}) = J_{P_{\ell-1}}, 
\quad \text{and} \quad (J_{P_{\ell-1}} : f_{\alpha_{1} \gamma_1}) \neq J_{P_{\ell-1}}.
$$
Combining these relations with Lemmas \ref{Reg_seq} and \ref{reg-grade}, it follows that $I_0$ is a deviation two ideal in its corresponding polynomial ring. Furthermore, by applying arguments analogous to those used in the case of $G_{1,(\ell,m,n)}$, we conclude that the entire family $G_{2,(\ell,m,n)}$ corresponds to binomial edge ideals satisfying the defining properties of Class I.
\end{proof}

\begin{rem}
It is shown in \cite{AN_NK24,ANNK24} that the edge binomials corresponding to the graphs $G_{1,(\ell,m,n)}$ and $G_{2,(\ell,m,n)}$ each form a $d$-sequence. Consequently, by \cite[Theorem 4.1]{HSV2} and \cite[Corollary 2.2]{HSV1}, the associated approximation complex $\mathcal{Z}_{\bullet}$ (see the next section for the definition) of the binomial edge ideal is acyclic.
\end{rem}

\section{Bigraded Free Resolutions of Symmetric algebra of Deviation $2$ ideals} \label{Section4}
The structure of symmetric algebras and their resolutions has been extensively studied for ideals of deviation zero \cite{MR1971,Eisenbud_text}, deviation one \cite{NKCV3, Ulrich&Polini}, and perfect ideals of deviation two \cite{Ulrich&Polini}. In this section, we extend this study to non-perfect ideals of deviation two.

Let $B=K[X]$ be a standard graded polynomial ring in the variables $X$, and $I$ be a graded ideal of $B$ minimally generated by $\{f_1, \ldots, f_n\}$ in degrees $\{d_1, \ldots,d_n\}$ respectively. Furthermore, let $S=K[X,Y]$ be a bigraded polynomial ring, where $Y=[y_1 \cdots y_{n}]$ with $\deg x_i=(1,0)$ and $\deg y_j=(d_j,1)$. Then the approximation complex $\mathcal{Z}_{\bullet}(I,B)$ is defined as follows:

\begin{equation}\label{eq:ApproxComplex}
\begin{aligned}
\mathcal{Z}_\bullet(I,B):\quad
0 &\longrightarrow Z_{n-1}\otimes_B S(-(n-1))
\xrightarrow{\partial_{n-1}} \cdots
\xrightarrow{\partial_4} Z_{3}\otimes_B S(-3)
\xrightarrow{\partial_{3}} Z_{2}\otimes_B S(-2)
\xrightarrow{\partial_{2}} \\
&\quad Z_{1}\otimes_B S(-1)
\xrightarrow{\partial_{1}} Z_{0}\otimes_B S
\longrightarrow 0
\end{aligned}
\end{equation}
where $Z_i$, for $i=1,\ldots,n$, denotes the $i^{\text{th}}$ Koszul cycle of $I$ and the differentials are induced by the Koszul maps with respect to the sequence $\{y_1 \ldots, y_n\}$ \cite{HSV1}. Since the zeroth homology of this complex is the symmetric algebra of $I$, the acyclicity of the complex implies that it provides a resolution of the symmetric algebra.

Now, assume $I$ to be an ideal of $B$ of deviation $2$ such that the corresponding approximation complex $\mathcal{Z}_\bullet(I,B)$ is acyclic.
Let $(\mathcal{F}_{\bullet},\varphi_{\bullet})$ denote a graded free resolution of $B/I$ of the following form,
\begin{equation}
\cdots \longrightarrow \bigoplus_{j \geq 3}B(-j)^{\beta_{3j}}\stackrel {\varphi_3} \longrightarrow \bigoplus_{j \geq 2}B(-j)^{\beta_{2j}} \stackrel{\varphi_2}\longrightarrow \bigoplus_{j=1}^{n}B(-d_j) \stackrel{\varphi_1} \longrightarrow B \longrightarrow 0. 
\end{equation}
Assume the following to be the Koszul complex, $(\mathcal{V}_\bullet,\nu_\bullet)$ corresponding to $\varphi_1$,

$$0 \longrightarrow {V}_{n} \stackrel{\nu_{n}}\longrightarrow \cdots \stackrel{\nu_{4}}\longrightarrow {V}_{3}   \stackrel{\nu_{3}} \longrightarrow {V}_{2} \stackrel{\nu_{2}} \longrightarrow {V}_{1} \stackrel{\varphi_{1}}\longrightarrow B \longrightarrow 0$$
where ${V}_i=\wedge^iB^n$ for $i=1, \ldots, n$.

Since the difference between the minimal number of generators of $I$ with its height is $2$, the corresponding Koszul homologies, $H_i(I,B)$ vanishes for $i \geq 3$, where $H_i(I,B)$ denotes the $i$-th homology of the Koszul complex. This implies that a free resolution of the Koszul cycles $Z_i$ for $i \geq 3$ can be obtained from the Koszul complex, as shown in the following diagram. Note that the dotted lines here indicate inclusions.

\[
\begin{tikzcd}
0 \arrow[r] & {V}_{n}  \arrow[r, "\nu_{n}"] & \cdots \arrow[r, "\nu_6"] & {V}_5 \arrow[r, "\nu_5"] \arrow[d, two heads] & {V}_4 \arrow[r, "\nu_4"] \arrow[d, two heads] & {V}_3 \arrow[r, "\nu_3"]  & {V}_2 \arrow[r, "\nu_2"]  & {V}_1 \arrow[r]& B \arrow[r] & 0 \\
&  & & Z_4 \arrow[ur, dashrightarrow] & Z_3 \arrow[ur, dashrightarrow] &  &  & & 
\end{tikzcd}
\]
We then obtain the following double complex, where $\partial_\mathbf{y}$ denotes the differentials induced by the Koszul maps with respect to the sequence $\{y_1, \ldots, y_n\}$, and $\partial_\mathbf{f}$ denotes the differentials induced by the Koszul maps with respect to the sequence $\{f_1, \ldots, f_n\}$,

\begin{equation} \label{double_complex}
\begin{tikzpicture}[baseline=(current  bounding  box.center)]
			
                \matrix (m) [matrix of math nodes,row sep=2.5em,column sep=2.3em,minimum width=2em, text height=1.8ex, text depth=0.25ex]				{ 
					& \vdots & \vdots & \vdots  \\
					0 & \wedge^{n}B^{n} \otimes_BS(-(n-2)) & \cdots & \wedge^{5}B^{n} \otimes_BS(-3) \\
					\wedge^{n}B^{n} \otimes_BS(-(n-1)) & \wedge^{n-1}B^{n} \otimes_BS(-(n-2)) & \cdots & \wedge^{4}B^{n} \otimes_BS(-3)  \\
				Z_{n-1} \otimes_B S(-(n-1)) & Z_{n-2} \otimes_B S(-(n-2)) & \cdots & Z_{3} \otimes_B S(-3) \\
                0 & 0 &  &  0  \\
                };						
				\path[-stealth]

                
                (m-1-2) edge node [right] {} (m-2-2)
				(m-1-4) edge node [right] {} (m-2-4)
                (m-2-1) edge node [right] {} (m-3-1)
                (m-2-2) edge node [right] {$\partial_\mathbf{f}$} (m-3-2)
				(m-2-4) edge node [right] {$\partial_\mathbf{f}$} (m-3-4)
                (m-3-1) edge node [right] {$\partial_\mathbf{f}$} (m-4-1)
                (m-3-2) edge node [right] {$\partial_\mathbf{f}$} (m-4-2)
				(m-3-4) edge node [right] {$\partial_\mathbf{f}$} (m-4-4)
                (m-4-1) edge node [right] {} (m-5-1)
                (m-4-2) edge node [right] {} (m-5-2)
				(m-4-4) edge node [right] {} (m-5-4)

                (m-2-1) edge node [above] {} (m-2-2)
                (m-2-2) edge node [above] {$\partial_\mathbf{y}$} (m-2-3)
				(m-2-3) edge node [above] {$\partial_\mathbf{y}$} (m-2-4)
                
                (m-3-1) edge node [above] {$\partial_\mathbf{y}$} (m-3-2)
                (m-3-2) edge node [above] {$\partial_\mathbf{y}$} (m-3-3)
				(m-3-3) edge node [above] {$\partial_\mathbf{y}$} (m-3-4)
               
                 (m-4-1) edge node [above] {$\partial_\mathbf{y}$} (m-4-2)
                (m-4-2) edge node [above] {$\partial_\mathbf{y}$} (m-4-3)
				(m-4-3) edge node [above] {$\partial_\mathbf{y}$} (m-4-4)

				;
                \end{tikzpicture}
\end{equation}

Let $(\mathcal{C}_\bullet, \sigma_\bullet)$ denote the total complex of the truncation of the double complex (\ref{double_complex}), with modules of the form 
\begin{equation} \label{truncated_complex}
C_i= \oplus_{j=0}^{i}\wedge^{4+i}B^{n} \otimes_BS(-j-3)
\end{equation}
for $i=0, \ldots, n-4$. Then from Equation (\ref{truncated_complex}), modules $C_i$, for $i=0, \ldots, n-4$, can be viewed as bishifted graded modules (with respect to the bigrading of $S$) of the following form,

\begin{equation}\label{bishifts}
C_i= \bigoplus_{j=0}^{i} \left[ \bigoplus_{1 \leq j_1 < j_2< \cdots <j_{i+4} \leq n} S(-(d_{j_1}+ d_{j_2}+ \cdots + d_{j_{i+4}}),-j-3)\right].
\end{equation}

From \cite[Proposition 4.5]{Ulrich&Polini} which states that the following complex  $\mathcal{J}_{\bullet}$ is acyclic whenever the corresponding complex $\mathcal{Z}_{\bullet}$ is acyclic, we obtain that the complex below resolves the symmetric algebra of $I$, where the tail part consists of bigraded free-modules.
\begin{equation}\label{NewComplex}
\mathcal{J}_\bullet:\quad 0 \longrightarrow C_{n-4} \stackrel{\sigma_{n-4}} \longrightarrow C_{n-3} \stackrel{\sigma_{n-3}}\longrightarrow \cdots \stackrel{\sigma_1} \longrightarrow C_{0}   \stackrel{{\eta}} \longrightarrow Z_{2} \otimes_B S(-2) \stackrel{} \longrightarrow  Z_{1} \otimes_B S(-1) \stackrel{}\longrightarrow S \longrightarrow 0.
\end{equation}
With the above building blocks in place, we now address the resolution of the symmetric algebras of ideals in Class I and Class II, using the ideas discussed so far.

Consider a family of ideals of the form in Class I or Class II. Set the following notations.
\begin{enumerate}[a)]
    \item Let $I_0=\langle f_1, \ldots, f_n \rangle$ with $\deg f_i=d_i$ for $1 \leq i \leq n$.
    \item Let the degree of $g_i$ be denoted by $d_{n+i}$ for $i \geq 1$.
    
    \item For $\alpha=\{i_1, i_2, \ldots,i_{g_{\alpha}}\} \subset [k]=\{1, \ldots, k\}$, $g_{\alpha} \leq k$, with $i_1 < i_2 <  \cdots < i_{g_{\alpha}}$, set $d_{\alpha}=\sum_{s=1}^{g_{\alpha}} d_{n+i_s}$ and let $|\alpha|$ denote the number of elements in the subset $\alpha$. 
 \item $(\mathcal{F}^{(k)}_\bullet, \varphi^{(k)}_\bullet)$ denotes a graded free resolution of $B/I_{k}$.
    
    \item $(\mathcal{V}^{(k)}_{\bullet},\nu^{(k)}_\bullet)$ denotes the Koszul complex corresponding to the generators of $I_k$.

    \item $(\mathfrak{L}^{(k)}_{\bullet},\tau^{(k)}_\bullet)$ denotes a graded free resolution of $H_2(I_k,B)$, the second homology of the Koszul complex corresponding to the generators of $I_k$.

    \item $(\mathcal{C}^{(k)}_\bullet, \sigma^{(k)}_\bullet)$ denote the total complex of the truncation of the double complex (\ref{double_complex}), with modules of the form given in Equation (\ref{truncated_complex}) which comes with respect to the Koszul complex on $I_k$.
\end{enumerate}
With the above notations, the following result gives the non-standard bigraded free resolutions of symmetric algebra of $I_k$.

\begin{thm} \label{RES}

Let $\{I_k\}_{k \geq 0}$ be a family of ideals of the form described in \emph{Class I} or \emph{Class II} in the standard graded polynomial ring $B = K[X]$. 
Consider the following graded free complexes:
\begin{enumerate}[a)]
    \item A graded free resolution of $I_0=\langle f_1, \ldots, f_n \rangle$,
    \begin{equation} \label{res_I0}
    \mathcal{F}^{(0)}_{\bullet} : \quad \cdots 
    \stackrel{\varphi^{(0)}_3}{\longrightarrow}  \bigoplus_{s \geq 2} B(-s)^{\beta_{(2,\,s)}} 
    \stackrel{\varphi^{(0)}_2}{\longrightarrow} \bigoplus_{s \geq 1} B(-s)^{\beta_{(1,\,s)}} 
    \stackrel{\varphi^{(0)}_1}{\longrightarrow} B \longrightarrow 0,
    \end{equation}
    where $\beta_{(i,\,s)}$ denotes the rank of the $i^{\text{th}}$ free module in degree shift $s$.
    
    \item A graded free resolution of $H_2(I_0,B)$,
    \begin{equation} \label{res_H2}
    \mathfrak{L}^{(0)}_{\bullet} : \quad \cdots 
    \stackrel{\tau^{(0)}_3}{\longrightarrow}  \bigoplus_{m \geq 2} B(-m)^{\beta'_{(2,\,m)}} 
    \stackrel{\tau^{(0)}_2}{\longrightarrow}  \bigoplus_{m \geq 1} B(-m)^{\beta'_{(1,\,m)}}  
    \stackrel{\tau^{(0)}_1}{\longrightarrow} B \longrightarrow 0,
    \end{equation}
    where $\beta'_{(i,\,m)}$ denotes the rank of the $i^{\text{th}}$ free module in degree shift $m$.
    
    \item The graded Koszul complex associated to $I_0$,
    \begin{equation} \label{res_Koszul}
    \mathcal{V}^{(0)}_{\bullet} : \quad \cdots 
    \stackrel{\nu^{(0)}_3}{\longrightarrow}  \bigoplus_{p \geq 2} B(-p)^{\beta''_{(2,\,p)}} 
    \stackrel{\nu^{(0)}_2}{\longrightarrow}  \bigoplus_{p \geq 1} B(-p)^{\beta''_{(1,\,p)}}  
    \stackrel{\nu^{(0)}_1}{\longrightarrow} B \longrightarrow 0,
    \end{equation}
    where $\beta''_{(i,\,p)}$ denotes the rank of the $i^{\text{th}}$ free module in degree shift $p$.
\end{enumerate}

Assume that the associated approximation complex $\mathcal{Z}_{\bullet}(I_k,B)$ is acyclic.  
Let $S = K[X,Y]$ with $Y = [y_1, \ldots, y_{n+
\mu(I_k)}]$, graded by $\deg X_i = (1,0)$ and $\deg Y_j = (d_{n+j},1)$ for $1 \leq j \leq n+k$.  
Then there exists a non-standard bigraded free resolution of the symmetric algebra $\mathrm{Sym}(I_k)$ over $S$ of the form given below.

$$\mathcal{D}^{(k)}_{\bullet} : \quad \cdots \longrightarrow {D}^{(k)}_2 \stackrel{\delta_2^{(k)}}\longrightarrow {D}^{(k)}_1 \stackrel{\delta_1^{(k)}}\longrightarrow {D}_0^{(k)} \longrightarrow 0$$

where, $${D}_0^{(k)}=S,$$ $${D}^{(k)}_1= \bigoplus_{s \geq 2}S(-s,-
1)^{\beta_{(2,s)}} \oplus \left[ \bigoplus_{i=0}^{k-1} \left[ \bigoplus_{\alpha \subset [k], \,  |\alpha|=i+1}  \left[ \bigoplus_{s \geq 1-i} S(-(s+d_{\alpha}),-1)^{\beta_{(1-i,s)}} \right] \right] \right],$$
\vspace{2mm}
and for $\ell \geq 2$,
\vspace{2mm}
\begin{align*}
{D}^{(k)}_\ell= & \bigoplus_{s \geq \ell+1}S(-s,-1)^{\beta_{(\ell+1,s)}} \oplus \left[ \bigoplus_{i=0}^{k-1} \left[ \bigoplus_{\alpha \subset [k], \,  |\alpha|=i+1}  \left[ \bigoplus_{s \geq \ell-i} S(-(s+d_{\alpha}),-1)^{\beta_{(\ell-i,s)}} \right] \right] \right]  \\
&  \text{ \hspace{40mm}}\textcolor{blue}{\bigoplus} \\
 &   \bigoplus_{p \geq \ell+1}S(-p,-2)^{\beta''_{(\ell+1,p)}} \oplus \left[ \bigoplus_{i=0}^{k-1} \left[ \bigoplus_{\alpha \subset [k], \,  |\alpha|=i+1}  \left[ \bigoplus_{p \geq \ell-i} S(-(p+d_{\alpha}),-2)^{\beta''_{(\ell-i,p)}} \right] \right] \right]\\
 &  \text{ \hspace{40mm}}\textcolor{blue}{\bigoplus} \\
 &\bigoplus_{m \geq \ell-2}S(-m,-2)^{\beta'_{(\ell-2,m)}} \oplus \left[ \bigoplus_{i=0}^{k-1} \left[ \bigoplus_{\alpha \subset [k], \,  |\alpha|=i+1}  \left[ \bigoplus_{m \geq \ell-i-3} S(-(m+d_{\alpha}),-2)^{\beta'_{(\ell-i-3,m)}} \right] \right] \right] \\
  &  \text{ \hspace{40mm}} \textcolor{blue}{\bigoplus} \\
  &  \text{ \hspace{15mm}} \bigoplus_{j=0}^{\ell-3}\left[ \bigoplus_{1 \leq i_1< i_2< \cdots < i_{\ell+1} \leq n+k } S(-(d_{i_1}+ d_{i_2}+ \cdots + d_{i_{\ell+1}}),-j-3)\right].
\end{align*} 
Moreover, the differentials $\delta_\ell^{(k)}$ have the standard block form induced by the mapping cone constructions.
\end{thm} 

\begin{proof}
Consider the $\mathcal{Z}$-complex corresponding to $I_k$,

\begin{align*}
\mathcal{Z}_\bullet(I_k,B): \hspace{2 mm} 0 \longrightarrow Z_{n-1}(I_k,B) \otimes S(-(n-1))  \stackrel{\partial_{n-1}^{(k)}}\longrightarrow \cdots \stackrel{\partial_{4}^{(k)}}\longrightarrow Z_{3}(I_k,B) \otimes S(-3)  \stackrel{\partial_{3}^{(k)}} \longrightarrow Z_{2}(I_k,B) \otimes S(-2) \stackrel{\partial_{2}^{(k)}} \longrightarrow  \\
Z_{1}(I_k,B) \otimes S(-1)\stackrel{\partial_{1}^{(k)}}\longrightarrow  Z_0(I_k,B) \otimes S \longrightarrow 0
\end{align*}

where $Z_i(I_k,B),$ $i=1, \ldots, n$, denotes the $i^{\text{th}}$ Koszul cycle of $I_k$.

From the discussions at the beginning of this section, the following complex provides a resolution of the symmetric algebra of $I_k$.
\begin{equation*}\label{NewComplexv2}
0 \longrightarrow C_{n-4}^{(k)} \stackrel{\sigma_{n-4}^{(k)}} \longrightarrow C_{n-3}^{(k)} \stackrel{\sigma_{n-3}^{(k)}}\longrightarrow \cdots \stackrel{\sigma_1^{(k)}} \longrightarrow C_{0}^{(k)}   \stackrel{\eta^{(k)}} \longrightarrow Z_{2}(I_k,B) \otimes_B S(-2) \stackrel{\partial_2^{(k)}} \longrightarrow  Z_{1}(I_k,B) \otimes_B S(-1) \stackrel{\partial_1^{(k)}}\longrightarrow S \longrightarrow 0.
\end{equation*}

Consider the short exact sequence,
\begin{equation} \label{SES1}
0 \longrightarrow \operatorname{coker}(\sigma_1^{(k)})   \stackrel{\eta^{(k)}} \longrightarrow Z_{2}(I_k,B) \otimes_B S(-2) \stackrel{\partial_2^{(k)}} \longrightarrow \operatorname{im}(\partial_2^{(k)})
\longrightarrow 0.
\end{equation}

Using Proposition \ref{H1-regular} and Lemma \ref{Res_of_H2}, for each $0 \leq j < k$, $$H_2(I_{j+1}, B) \cong H_2(I_{j}, B)/  g_{j+1}  H_2(I_{j}, B).$$  Iteratively applying the mapping cone construction to the short exact sequences
\[ 0 \longrightarrow H_2(I_j, B)(-d_{j+1}) \xrightarrow{\cdot g_{j+1}} H_2(I_j, B) \longrightarrow H_2(I_{j+1}, B) \longrightarrow 0, \]

gives a resolution of $H_2(I_k, B)$ in terms of the resolution of $H_2(I_0, B)$.

As a consequence, we obtain the following explicit graded free resolution of $H_2(I_k,B)$, 
$$ \cdots \longrightarrow \mathfrak{L}^{(k)}_2 \longrightarrow \mathfrak{L}^{(k)}_1 \longrightarrow \mathfrak{L}^{(k)}_0 \longrightarrow 0$$

where for $\ell \geq 0$, $$\mathfrak{L}^{(k)}_\ell=\bigoplus_{m \geq \ell}B(-m)^{\beta'_{(\ell,m)}} \oplus \left[ \bigoplus_{i=0}^{k-1} \left[ \bigoplus_{\alpha \subset [k], \,  |\alpha|=i+1}  \left[ \bigoplus_{m \geq \ell-i-1} B(-(m+d_{\alpha}))^{\beta'_{(\ell-i-1,m)}} \right] \right] \right]$$
 
and the differentials which come corresponding to the iterative mapping cone construction take the following form.

For $\ell >1$,
$$
\tau_{\ell}^{(k)} = \begin{pmatrix}
- \tau_{\ell}^{(k-1)} & g_\ell \cdot \mathrm{id}_{{L}_{\ell-1}^{(k-1)}} \\
0 & \tau_{\ell-1}^{(k-1)}
\end{pmatrix},
$$
and
$$
\tau_1^{(k)} = \begin{pmatrix}
\tau_1^{(k-1)} &
g_\ell \cdot \mathrm{id}_{{L}_0^{(k-1)}} \\
\end{pmatrix}.
$$

Now, consider the short exact sequence
$$
0 \longrightarrow B_3(I_k, B) \longrightarrow Z_2(I_k, B) \longrightarrow H_2(I_k, B) \longrightarrow 0,
$$
and observe that the truncated Koszul complex
$$
0 \longrightarrow {V}_{n}^{(k)}\stackrel{\nu^{(k)}_n}  \longrightarrow \cdots \stackrel{\nu^{(k)}_5} \longrightarrow {V}_{4}^{(k)} \stackrel{\nu^{(k)}_4} \longrightarrow {V}_{3}^{(k)} \longrightarrow 0,
$$
where, for $\ell \geq 3$,
$$
{V}^{(k)}_\ell = \bigoplus_{p \geq \ell} B(-p)^{\beta''_{(\ell,p)}} \oplus \left[ \bigoplus_{i=0}^{k-1} \left[ \bigoplus_{\substack{\alpha \subset [k], \, |\alpha| = i+1}} \left[ \bigoplus_{p \geq \ell-i-1} B(-(p + d_{\alpha}))^{\beta''_{(\ell-i-1,p)}} \right] \right] \right],
$$
with $$
\nu_{\ell}^{(k)} = \begin{pmatrix}
- \nu_{\ell}^{(k-1)} & g_\ell \cdot \mathrm{id}_{{K}_{\ell-1}^{(k-1)}} \\
0 & \nu_{\ell-1}^{(k-1)}
\end{pmatrix},
$$
is a graded free resolution of $B_3(I_k,B)$.
Then, by applying the Horseshoe Lemma to this short exact sequence, one obtains a graded free resolution $\mathcal{G}^{(k)}_\bullet$ of $Z_2(I_k,B)$, where, for $\ell \geq 0$, ${G}_\ell^{(k)}={K}_{\ell+3}^{(k)} \bigoplus {L}_\ell^{(k)}$ and the corresponding differentials are denoted by $\epsilon_{\ell}^{(k)}$.

Tensoring with $S(-2)$, we obtain a bigraded resolution $\tilde{\mathcal{G}}^{(k)}_\bullet$ of $Z_2(I_k,B) \otimes S(-2)$ whose differentials, denoted by $\tilde{\epsilon}_{\ell}^{(k)}$, 
are induced from $\epsilon_{\ell}^{(k)} \otimes \mathrm{id}_{S(-2)}$. Taking the mapping cone of the map $\eta^{(k)}$ in Equation \eqref{SES1}, we get a bigraded free resolution of $\operatorname{im}(\partial_2^{(k)})$
\[ (\text{Cone}(\eta^{(k)}))_\ell =  ({K}_{\ell+3}^{(k)} \oplus {L}_\ell^{(k)}) \otimes S(-2) \oplus  {C}_{\ell-1}^{(k)} \]

with differentials,
$$
\rho_{\ell}^{(k)} = \begin{pmatrix}
- \tilde\epsilon_{\ell}^{(k)} & \eta^{(k)} \\
0 &    \sigma_{\ell-1}^{(k)}
\end{pmatrix}.
$$

Now, consider the following short exact sequence, 
\begin{equation} \label{SES2}
0 \longrightarrow \operatorname{im}(\partial_2^{(k)})   \stackrel{\eta'^{(k)}} \longrightarrow Z_{1}(I_k,B) \otimes_B S(-1) \stackrel{\partial_1^{(k)}} \longrightarrow \operatorname{im}(\partial_1^{(k)})
\longrightarrow 0
\end{equation}
where $\eta'^{(k)}$ is the inclusion map.

Since $Z_1(I_k,B)$ is isomorphic to the first syzygy of $I_k$, we have the following truncated complex is a graded free resolution of $Z_1(I_k,B)$, 
$$ \cdots \longrightarrow {F}^{(k)}_4 \stackrel{\varphi^{(k)}_4} \longrightarrow {F}^{(k)}_3 \stackrel{\varphi^{(k)}_3} \longrightarrow {F}^{(k)}_2 \longrightarrow 0$$

where for $\ell \geq 2$, $${F}^{(k)}_{\ell}=\bigoplus_{s \geq \ell}B(-s)^{\beta_{(\ell,s)}} \oplus \left[ \bigoplus_{i=0}^{k-1} \left[ \bigoplus_{\alpha \subset [k], \,  |\alpha|=i+1}  \left[ \bigoplus_{s \geq \ell-i-1} B(-(s+d_{\alpha}))^{\beta_{(\ell-i-1,s)}} \right] \right] \right]$$

and the differentials are given by, 
$$
\varphi_{\ell}^{(k)} = \begin{pmatrix}
- \varphi_{\ell}^{(k-1)} & g_\ell \cdot \mathrm{id}_{{F}_{\ell-1}^{(k-1)}} \\
0 & \varphi_{\ell-1}^{(k-1)}
\end{pmatrix}.
$$

Denote the graded free resolution of $Z_1(I_k, B)$ by ($\mathcal{H}^{(k)}_\bullet$,$\kappa^{(k)}_\bullet$), where for $\ell \geq 0$, ${H}^{(k)}_\ell = {F}^{(k)}_{\ell+2}$. Let $\tilde{\mathcal{H}}^{(k)}_\bullet$ denote the corresponding bigraded free resolution of $Z_1(I_k, B) \otimes S(-1)$, given by $\tilde{{H}}^{(k)}_\ell = {F}^{(k)}_{\ell+2} \otimes S(-1)$ with induced differentials denoted by $\tilde\kappa^{(k)}_\ell$ for all $\ell \geq 0$.

Taking the mapping cone of $\eta'^{(k)}$ in Equation~\eqref{SES2}, we finally obtain the bigraded free resolution of the defining ideal of the Rees algebra of $I_k$:
$$(\text{Cone}(\eta'^{(k)}))_0= \tilde{H}_{0}^{(k)}={F}_{2}^{(k)}  \otimes S(-1),$$
and for $\ell \geq 1$,
$$(\text{Cone}(\eta'^{(k)}))_\ell=\tilde{H}_{\ell}^{(k)}  \oplus    \text{Cone}(\eta^{(k)})_{\ell-1}   ={F}_{\ell+2}^{(k)}  \otimes S(-1)   \oplus \left( ({K}_{\ell+2}^{(k)} \oplus {L}_{\ell-1}^{(k)}) \otimes S(-2) \oplus {C}_{\ell-2}^{(k)} \right).$$
The differentials are given by, $\delta_1^{(k)}= \tilde\kappa_{\ell}^{(k)}$, 
and for $\ell \geq 1$
$$
\delta_{\ell}^{(k)} = \begin{pmatrix}
- \tilde\kappa_{\ell}^{(k)}  & \eta'^{(k)} \\
0 & \rho_{\ell-1}^{(k)} 
\end{pmatrix}.
$$

\end{proof}

We conclude with an elementary but useful observation.

\begin{rem}
If, in addition to the hypotheses of Theorem \ref{RES}, the ideal $I$ is of linear type, then Theorem \ref{RES} provides an explicit description of the bigraded free resolution of the corresponding Rees algebra.
\end{rem}

\noindent{\bf Acknowledgment.} The first author gratefully acknowledges partial support from the Anusandhan National Research Foundation (ANRF), Government of India, under the Core Research Grant (File Number: CRG/2023/007668). The second author acknowledges support from the University Grants Commission (UGC), Government of India.

\vspace{3mm}

\noindent{\bf Statements and Declarations.} 

\vspace{2mm}

\noindent{\bf Competing Interests.} On behalf of all authors, the corresponding author states that there is no conflict of interest.

\bibliographystyle{abbrv}

\bibliography{references}

\end{document}